\documentclass[11pt,reqno,sumlimits]{amsart}

\usepackage{amssymb, amscd, amsmath, epsfig, mathtools}
\usepackage{amsthm}
\usepackage{enumerate}
\usepackage{xcolor}
\usepackage{scalerel}
\usepackage{soul}
\usepackage{tikz-cd}

\usepackage[margin=1.0in]{geometry}

\newtheorem{theorem}{Theorem}[section]
\newtheorem*{theorem*}{Theorem}
\newtheorem{definition}{Definition}[section]
\newtheorem{corollary}{Corollary}[section]

\newtheorem{proposition}{Proposition}[section]
\theoremstyle{definition}
\newtheorem{remark}{Remark}[section]

\newcommand{\R}{\mathbb R}

\newcommand{\calC}{\mathcal C}
\newcommand{\calL}{\mathcal L}
\newcommand{\dvol}{ d\text{Vol}_{g}}

\begin{document}

\title[The Boundary Yamabe Problem with Minimal Boundary Case]{The Boundary Yamabe Problem, I: Minimal Boundary Case}
\author[J. Xu]{Jie Xu}
\address{
Department of Mathematics and Statistics, Boston University, Boston, MA, USA}
\email{xujie@bu.edu}

\date{}	

\maketitle

\begin{abstract} We apply iteration schemes and perturbation methods to provide a complete solution of the boundary Yamabe problem with minimal boundary scenario, or equivalently, the existence of a real, positive, smooth solution of $ -\frac{4(n -1)}{n - 2} \Delta_{g} u + S_{g} u = \lambda u^{\frac{n+2}{n - 2}} $ in $ M $, $ \frac{\partial u}{\partial \nu} + \frac{n-2}{2} h_{g} u = 0 $ on $ \partial M $. Thus $ g $ is conformal to to the metric $ \tilde{g} = u^{\frac{4}{n -2}} g $ of constant scalar curvature $ \lambda $ with minimal boundary. In contrast to the classical method of calculus of variations with assumptions on Weyl tensors and classification of types of points on $ \partial M $, the boundary Yamabe problem is fully solved here in three cases classified by the sign of the first eigenvalue $ \eta_{1} $ of the conformal Laplacian with Robin condition. When $ \eta_{1} < 0 $, a pair of global sub-solution and super-solution are constructed. When $ \eta_{1} > 0 $, a perturbed boundary Yamabe equation $ -\frac{4(n -1)}{n - 2} \Delta_{g} u_{\beta} + \left( S_{g} + \beta \right) u_{\beta} = \lambda_{\beta} u_{\beta}^{\frac{n+2}{n - 2}} $ in $ M $, $ \frac{\partial u_{\beta}}{\partial \nu} + \frac{n-2}{2} h_{g} u_{\beta} = 0 $ on $ \partial M $ is solved with $ \beta < 0 $. The boundary Yamabe equation is then solved by taking $ \beta \rightarrow 0^{-} $. The signs of scalar curvature $ S_{g} $ and mean curvature $ h_{g} $ play important roles in this existence result. 
\end{abstract}

\section{Introduction}
In this article, we completely solve the boundary Yamabe problem for the minimal boundary case on compact manifolds $ (\bar{M}, g) $ with smooth boundary, $ \dim \bar{M} \geqslant 3 $ by an iteration scheme and a perturbation method. The iteration scheme is inspired by earlier works on either a local Riemannian domain $ (\Omega, g) $ \cite{XU2}, or on closed manifolds $ (M, g) $ \cite{XU3}. A similar iteration method is also used to solve Einstein vacuum equation \cite{Hintz}, \cite{HVasy} and nonlinear Laplace equation \cite{Xu}, with a long history in PDE theory dating back to \cite{Moser1, Moser2}. A modification of the monotone iteration method, due to \cite{SA}, is applied here for the Robin boundary condition on manifolds. In the most difficult case where the first eigenvalue $ \eta_{1} $ of the conformal Laplacian is positive, a perturbation method is introduced to solve the perturbed boundary Yamabe equation $ -\frac{4(n -1)}{n - 2} \Delta_{g} u_{\beta} + \left( S_{g} + \beta \right) u_{\beta} = \lambda_{\beta} u_{\beta}^{\frac{n+2}{n - 2}} $ in $ M $, $ \frac{\partial u_{\beta}}{\partial \nu} + \frac{n-2}{2} h_{g} u_{\beta} = 0 $ on $ \partial M $, which is based on a local solvability of perturbed Yamabe equation with Dirichlet boundary condition and monotone iteration scheme. One advantage of this local analysis is to bypass the role of Weyl tensor both in interior points of the manifolds, and to avoid the classification of boundary points and the vanishing of the Weyl tensors at boundary. In the second paper in this series, we solve the most general case by replacing the zero mean curvature condition with constant mean curvature.
\medskip

In 1992, Escobar \cite{ESC} proposed the following generalization of the classical Yamabe problem on closed manifolds, which is called the boundary Yamabe problem or Escobar problem, and is a far reaching generalization of the uniformization theorem for surfaces:
\medskip

{\bf The Boundary Yamabe Problem.} {\it Given a compact Riemannian manifold $ (\bar{M}, g) $ of dimension $ n \geqslant 3 $ with interior $ M $ and smooth boundary $ \partial M $, there exists a metric $ \tilde{g} $ conformal to $ g $ having constant scalar curvature and minimal boundary.}       
\medskip
  
Let $ S_{g} $ be the scalar curvature  of $g$ and $ h_{g} $ be the mean curvature on $ \partial M $, and let $ \tilde{S} $, $ \tilde{h} $ be the scalar curvature and mean curvature of the conformal metric $ \tilde{g} = e^{2f} g $, respectively. Let $ \nu $ be the outward normal vector field along $ \partial M $. Set $ e^{2f} = u^{p-2} $, where $ p = \frac{2n}{n - 2} $ and $ u >0 $. Then
\begin{equation}\label{intro:eqn1}
\begin{split}
\tilde{S} & = u^{1-p}\left(-4 \cdot \frac{n-1}{n-2} \Delta_{g} u + S_{g}u \right) \; {\rm in} \; M; \\
\tilde{h} & = e^{-f} \left( h_{g} + \frac{\partial f}{\partial \nu} \right) \; {\rm on} \; \partial M.
\end{split}
\end{equation}
\noindent 
Setting $ a = 4 \cdot \frac{n-1}{n-2} > 0 $, we have that 
$ \tilde{g} = u^{p-2} g $ has constant scalar curvature $ \lambda $ and minimal boundary if and only if $ u $ satisfies the boundary Yamabe equation  
\begin{equation}\label{intro:eqn2}
\begin{split}
\Box_{g}u & : =  -a\Delta_{g} u + S_{g} u = \lambda u^{p-1} \; {\rm in} \; M; \\
B_{g} u & : = \frac{\partial u}{\partial \nu}  + \frac{2}{p-2} h_{g} u = 0 \; {\rm on} \; \partial M.
\end{split}
\end{equation}
where $\Delta_g = -d^*d$ is negative definite.
\medskip

If the requirement that $ \partial M $ is minimal with respect to $ \tilde{g} $ is dropped, a nontrivial mean curvature $ \tilde{h} $ has to be introduced. As Escobar mentioned in \cite{ESC2}, the general boundary Yamabe problem on compact Riemannian manifold with smooth boundary $ (M, g) $ is equivalent to solve the following PDE
\begin{equation}\label{intro:eqn3}
\begin{split}
& -a\Delta_{g} u + S_{g} u = \lambda u^{p-1} \; {\rm in} \; M; \\
& \frac{\partial u}{\partial \nu} = \frac{2}{p-2} \left( - h_{g} u + \tilde{h} u^{\frac{p}{2}} \right) \; {\rm on} \; \partial M.
\end{split}
\end{equation}
Here $ \lambda $ is the constant scalar curvature of $ \tilde{g} $ and $ \tilde{h} $ is the constant mean curvature on $ \partial M $ with respect to $ \tilde{g} $. Most cases of the boundary Yamabe problem with minimal boundary condition have been handled in works of \cite{Brendle}, \cite{Escobar2}, \cite{HL}, etc. In addition, \cite{Marques} among others worked on the non-minimal case where $ \lambda = 0 $ and $ \tilde{h} $ is a constant. However, there are several cases of the minimal boundary scenario still left open. In \cite{ESC}, the unsolved cases are when $ n \geqslant 6 $, $ M $ is not locally conformally flat, $ \partial M $ is umbilic, and the Weyl tensor vanishes identically on $ \partial M $. This result was improved in \cite{CS} with some extra restrictions on the manifolds. For closed manifolds, \cite{Aubin, PL} provided good survey with classical calculus of variation methods, while a direct analysis can be found in \cite{XU3}. On non-compact manifolds, results with certain restrictions are in e.g. \cite{Aviles-McOwen, Grosse, PDEsymposium}.
\medskip

The main result of this article, which states below, provides a complete solution of boundary Yamabe problem.
\begin{theorem*}
Let $ (\bar{M}, g) $ be a compact manifold with smooth boundary, $ \dim \bar{M} \geqslant 3 $. Let $ \eta_{1} $ be the first eigenvalue of the boundary value problem $ \Box_{g} u = \eta_{1} u $ in $ M $, $ B_{g} u = 0 $ on $ \partial M $. Then
\begin{enumerate}[(i).]
\item If $ \eta_{1} = 0 $, then (\ref{intro:eqn2}) has a real, positive solution $ u \in \calC^{\infty}(\bar{M}) $ with $ \lambda = 0 $; 
\item If $ \eta_{1} < 0 $, then (\ref{intro:eqn2}) has a real, positive solution $ u \in \calC^{\infty}(\bar{M}) $ with $ \lambda < 0 $; 
\item If $ \eta_{1} > 0 $, then (\ref{intro:eqn2}) has a real, positive solution $ u \in \calC^{\infty}(\bar{M}) $ with $ \lambda > 0 $.
\end{enumerate}
\end{theorem*}
\medskip

Case (i) is a trivial case, since it's just an eigenvalue problem. Case (ii) is solved in Theorem \ref{yamabe:thm2} and \ref{yamabe:thm4}. We first get a solution of (\ref{intro:eqn2}) when $ h_{g} > 0 $ everywhere on $ \partial M $; then Theorem \ref{yamabe:thm3} says that the general case when $ \eta_{1} < 0 $ can be converted to the special case just mentioned. Case (iii) is solved in three steps: a perturbed boundary Yamabe equation $ -\frac{4(n -1)}{n - 2} \Delta_{g} u_{\beta} + \left( S_{g} + \beta \right) u_{\beta} = \lambda_{\beta} u_{\beta}^{\frac{n+2}{n - 2}} $ in $ M $, $ \frac{\partial u_{\beta}}{\partial \nu} + \frac{n-2}{2} h_{g} u_{\beta} = 0 $ on $ \partial M $ is solved in Theorem \ref{yamabe:thm5} with some $ \beta < 0 $; then the boundary Yamabe problem is solved in Theorem \ref{yamabe:thms} for $ \eta_{1} > 0 $, $ h_{g} > 0 $ everywhere on $ \partial M $ and $ S_{g} < 0 $ somewhere in $ M $;  lastly Theorem \ref{yamabe:thm7} shows that every general case for $ \eta_{1} > 0 $ reduces to the scenario in Theorem \ref{yamabe:thms}, due to Theorem \ref{yamabe:thm3} and \ref{yamabe:thm6}. Inspired by \cite{XU2} and \cite{XU3}, we prove a crucial local result with respect to a perturbation of $ \beta $ within a small enough interior domain $ \Omega \subset M $. As a crucial technical point, the small radius volume of geodesic balls is controlled by the scalar curvature, while the Weyl tensor does not influence the volume..
\medskip

In the classical calculus of variations approach, a solution of the boundary Yamabe problem is a minimizer of the functional
\begin{equation*}
Q(M) = \inf_{u \neq 0}  \frac{\int_{M} \left( a\lvert \nabla_{g} u \rvert^{2} + S_{g} u^{2} \right) \dvol + \int_{\partial M} \frac{2a}{p-2}h_{g} u^{2} dS}{\left( \int_{M} u^{p} \dvol \right)^{\frac{2}{p}}}
\end{equation*}
The existence of the minimizer relies heavily on showing $ Q(M) < Q(\mathbb{S}_{+}^{n}) $. The existence of of a minimizer is broken down into several cases, depending on whether Weyl tensor vanishes or not, and whether a boundary point is umbilic or not. In particular, the analysis of the Yamabe quotient near the boundary requires different test functions depending on the nature of the Weyl tensor, the existence of umbilic points, and the vanishing of Weyl tensor on $ \partial M $. In contrast, when $ \eta_{1} < 0 $ we apply the idea of Kazdan and Warner \cite{KW2} from the closed manifold case to construct global sub-solutions and super-solutions. Historically this is also an easy case.

The hard case is when $ \eta_{1} > 0 $. In the classical approach, subcritical solutions are constructed, i.e. the boundary Yamabe equations with subcritical nonlinear terms $ u^{s - 1}, s \in (2, p) $ are solved; then a limiting argument as $ s \rightarrow p^{-} $ is required. This limiting process as well as the proof of the positivity of the limit require $ Q(M) < Q(\mathbb{S}_{+}^{n}) $. In our method we bypass this subcritical argument by perturbing the coefficient of the zeroth order term of the differential operator, instead of perturbing the exponent $ p -1 $ of the nonlinear term. Fixing the exponent simplifies the limiting argument significantly, as only $ Q(M) \leqslant Q(\mathbb{S}_{+}^{n}) $ is required. In particular, we use the local to global analysis developed for closed manifolds \cite{XU3}: first, we construct a local solution of the perturbed Yamabe equation $ -a\Delta_{g} u + (S_{g} + \beta) u = \lambda u^{p-1} $, $ \beta < 0 $, in a small interior domain $ \Omega $ with trivial Dirichlet boundary condition; secondly, we apply monotone iteration scheme to obtain a global solution of perturbed Yamabe equation with Robin boundary condition $ B_{g} u = 0 $; finally, we pass to the limit $ \beta \rightarrow 0^{-} $ to obtain a solution. The local analysis in the first step is essential in the limiting argument.

Through the local analysis, iteration scheme and perturbation methods in this article and in \cite{XU3}, the Yamabe problem on closed manifolds and the Escobar problem on compact manifolds with boundary have a synchronized methodology: (i) the solvability in both cases are classified by the sign of the first eigenvalue $ \eta_{1} $ of conformal Laplacian only; (ii) when $ \eta_{1} < 0 $, both cases are solved by constructing a global subsolution and supersolution; (iii) when $ \eta_{1} > 0 $, both cases are solved by solving a perturbed PDE followed with a limiting argument which annihilates the perturbed term. In contrast, the classical arguments are asynchronous: historically the locally conformally flat cases are hardest cases on closed manifolds with dimensions $ n \geqslant 6 $; meanwhile the cases when $ \bar{M} $ is not locally conformally flat has been open on compact manifolds with boundary when $ \dim(\bar{M}) \geqslant 6 $.
\medskip

This article is organized as follows. In \S2, definitions and essential tools are listed and proved if necessary. In \S3, we first prove a global $ \calL^{p} $-regularity result  in Theorem \ref{global:thm1} for second order elliptic PDE with Robin condition by assuming the existence of the solution of this type of PDE. This work is based on a local result by Agmon, Douglis and Nirenberg \cite{Niren4}. A general elliptic estimate is also given. Assuming the injectivity of the second order elliptic operator, a specific $ \calL^{p} $ elliptic estimate is obtained in Theorem \ref{global:thm2}. These $ \calL^{p} $-regularity theory and $ \calL^{p} $-estimate are then used to prove the existence of the solution of $ -a\Delta_{g} u = F(x, u) $ with Robin condition by a monotone iteration method on compact manifolds with boundary in Theorem \ref{global:thm3}, provided the existence of corresponding linear elliptic PDE. In \S4, an existence theorem of the elliptic linear PDE $ -\Lambda \Delta_{g} u + \Lambda' u = f $ with Robin boundary condition is given. In \S5, the boundary Yamabe problem with minimal
boundary is fully solved in several steps. Corollary \ref{yamabe:cor1} handles the case $ \eta_{1} = 0 $; Theorem \ref{yamabe:thm2} and \ref{yamabe:thm4} handle the case $ \eta_{1} < 0 $; Theorem \ref{yamabe:thms} and \ref{yamabe:thm7} handle the case $ \eta_{1} > 0 $ by the crucial perturbation result in Theorem \ref{yamabe:thm5}. We end with some results, given in Theorem \ref{yamabe:thm3}, \ref{yamabe:thm6}, Corollary \ref{yamabe:cor2} and \ref{yamabe:cor3}, on when functions $ f_{1}, f_{2} $ can be the prescribed scalar curvature and mean curvature respectively, of a metric conformal to a given metric.
\section{The Preliminaries}
In this section, we list necessary definitions and results in order to solve this boundary Yamabe problem. Throughout this section, we consider the spaces with dimensions no less than $ 3 $.

Let $ \Omega $ be a connected, bounded, open subset of $ \R^{n} $ with smooth boundary $ \partial \Omega $ equipped with some Riemannian metric $ g $ that can be extended smoothly to $ \bar{\Omega} $. We call $ (\Omega, g) $ a Riemannian domain. Furthermore, let $ (\bar{\Omega}, g) $ be a compact manifold with boundary. 
\medskip

Firstly we define Sobolev space on compact manifolds $ (\bar{M}, g) $ with interior $ M $ and smooth boundary $ \partial M $. The integer ordered Sobolev spaces defined on $ (\bar{M}, g) $ is defined on $ (M, g) $, where $ M $ is the interior. We also define Sobolev spaces on Riemannian domain $ (\Omega, g) $.
\begin{definition}\label{boundary:def1} Let $ (\bar{M}, g) $ be a compact Riemannian manifold with smooth boundary $ \partial M $ and interior $ M $, let $ \dim M = n $. Let $ d\omega $ be the Riemannian density with local expression $ \dvol $. Let $ dS $ be the induced boundary density on $ \partial M $. For real valued functions $ u $, we set: 

(i) 
For $1 \leqslant p < \infty $, 
\begin{align*}
\mathcal{L}^{p}(M, g)\ &{\rm is\ the\ completion\ of}\ \left\{ u \in \calC_{c}^{\infty}(\bar{M}) : \Vert u\Vert_{p,g}^p :=\int_{M} \left\lvert u \right\rvert^{p} d\omega < \infty \right\}; \\
\mathcal{L}^{p}(\Omega, g)\ &{\rm is\ the\ completion\ of}\ \left\{ u \in \calC_c^{\infty}(\Omega) : \Vert u\Vert_{p,g}^p :=\int_{\Omega} \left\lvert u \right\rvert^{p} d\text{Vol}_{g} < \infty \right\}.
\end{align*}

(ii)
For $\nabla$ the Levi-Civita connection of $g$, and for 
$ u \in \calC^{\infty}(M) $,
\begin{equation}\label{boundary:eqn1}
\lvert \nabla^{k} u \rvert_g^{2} := (\nabla^{\alpha_{1}} \dotso \nabla^{\alpha_{k}}u)( \nabla_{\alpha_{1}} \dotso \nabla_{\alpha_{k}} u).
\end{equation}

In particular, $ \lvert \nabla^{0} u \rvert^{2}_g = \lvert u \rvert^{2}_g $ and $ \lvert \nabla^{1} u \rvert^{2}_g = \lvert \nabla u \rvert_{g}^{2} $.

(iii) For $ s \in \mathbb{N}, 1 \leqslant p < \infty $,
\begin{equation}\label{boundary:eqn2}
\begin{split}
W^{s, p}(M, g) & = \left\{ u \in \mathcal{L}^{p}(M, g) : \lVert u \rVert_{W^{s, p}(M, g)}^{p} = \sum_{j=0}^{s} \int_{M} \left\lvert \nabla^{j} u \right\rvert^{p}_g d\omega < \infty \right\}; \\
W^{s, p}(\Omega, g) &= \left\{ u \in \mathcal{L}^{p}(\Omega, g) : \lVert u \rVert_{W^{s, p}(\Omega, g)}^{p} = \sum_{j=0}^{s} \int_{\Omega} \left\lvert \nabla^{j} u \right\rvert^{p}_g d\text{Vol}_{g} < \infty \right\}.
\end{split}
\end{equation}

Similarly, $ W_{0}^{s, p}(M, g) $ is the completion of $ \calC_{c}^{\infty}(M) $ with respect to the $ W^{s, p} $-norm. In particular, $ H^{s}(M, g) : = W^{s, 2}(M, g), s \in \mathbb{N}, 1 \leqslant p' < \infty $ are the usual Sobolev spaces, and we similarly define $H_{0}^{s}(M, g) $, $ W_{0}^{s, p}(\Omega, g) $ and $ H_{0}^{s}(\Omega, g) $.

(iv) With an open cover $ \lbrace U_{\xi}, \phi_{\xi} \rbrace $ of $ (\bar{M}, g) $ and a smooth partition of unity $ \lbrace \chi_{\xi} \rbrace $ subordinate to this cover, we can define the $ W^{s, p} $-norm locally, which is equivalent to the definition above.
\begin{equation*}
\lVert u \rVert_{W^{s, p}(M, g)} = \sum_{\xi} \lVert \left(\phi_{\xi}^{-1}\right)^{*} \chi_{\xi} u \rVert_{W^{s, p}(\phi_{\xi}(U_{\xi}), g)}.
\end{equation*}
\end{definition}
\medskip

Let's denote the conformal Laplacian with the boundary condition to be
\begin{equation}\label{boundary:eqn3}
\Box_{g}u : = -a\Delta_{g} u + S_{g} u, B_{g} u : = \frac{\partial u}{\partial \nu} + \frac{2}{p - 2} h_{g} u, \forall u \in \calC_{c}^{\infty}(M).
\end{equation}
Let's denote the first eigenvalue of $ \Box_{g} $ with boundary condition $ B_{g} u = 0 $ to be $ \eta_{1} $, which is characterized by
\begin{equation}\label{boundary:eqn4}
\eta_{1} = \inf_{u \neq 0} \frac{\int_{M} a \lvert \nabla_{g} u \rvert^{2} d\omega + \int_{M} S_{g} u^{2} d\omega + \int_{\partial M} \frac{2a}{p - 2}h_{g}u^{2} dS}{\int_{M} u^{2} d\omega}.
\end{equation}
The following result is needed due to Escobar \cite{ESC}.
\begin{proposition}\label{boundary:prop1}\cite[Prop.~1.3.]{ESC}
Let $ \tilde{g} = u^{p-2} g $ be a conformal metric to $ g $. Let $ \eta_{1} $ and $ \tilde{\eta}_{1} $ be the first eigenvalue of $ \Box_{g} $ and $ \Box_{\tilde{g}} $ with boundary conditions $ B_{g} = 0 $ and $ B_{\tilde{g}} = 0 $, respectively. Then either the signs of $ \eta_{1} $ and $ \tilde{\eta}_{1} $ are the same or $ \eta_{1} = \tilde{\eta}_{1} = 0 $.
\end{proposition}
\medskip

A local $ L^{p} $ regularity is required for some type of Robin boundary condition, due to Agmon, Douglis, and Nirenberg \cite{Niren4}.
\begin{proposition}\label{boundary:prop2}\cite[Thm.~7.3, Thm.~15.2]{Niren4} Let $ (\Omega, g) $ be a Riemannian domain where the boundary $ \partial \Omega $ satisfies Lipschitz condition. Let $ \nu $ be the outward unit normal vector along $ \partial \Omega $. Let $ L $ be the second order elliptic operator on $ \Omega $ with smooth coefficients up to $ \partial M $ and $ f \in \calL^{p}(\Omega, g) $, $ f' \in W^{1, p}(\Omega, g) $ for some $ p \in (1, \infty) $. Let $ u \in H^{1}(\Omega, g) $ be the weak solution of the following boundary value problem
\begin{equation}\label{boundary:eqn5}
L u = f \; {\rm in} \; \Omega, Bu : = \frac{\partial u}{\partial \nu} + c(x) u = f' \; {\rm on} \; \partial \Omega,
\end{equation}
where $ c \in \calC^{\infty}(\partial \Omega) $. Then $ u \in W^{2, p}(\Omega, g) $ and the following estimates holds provided $ u \in \calL^{p}(\Omega, g) $:
\begin{equation}\label{boundary:eqn6}
\lVert u \rVert_{W^{2, p}(\Omega, g)} \leqslant C^{*} \left( \lVert Lu \rVert_{\calL^{p}(\Omega, g)} + \lVert Bu \rVert_{W^{1, p}(\Omega, g)} +  \lVert u \rVert_{\calL^{p}(\Omega, g)} \right).
\end{equation}
Here the constant $ C^{*} $ depends on $ L, p $ and $ (\Omega, g) $.
\end{proposition}
\begin{remark}\label{boundary:re1}
 It is worth mentioning that the result in Proposition \ref{boundary:prop2} holds on a $ n $ dimensional hemisphere denoted by $ \sum_{i = 1}^{n - 1} x_{i}^{2} + t^{2} \leqslant 1, t \geqslant 0 $ where the boundary condition is only defined on $ t = 0 $ and $ u $ in (\ref{boundary:eqn5}) vanishes outside the hemisphere \cite[Thm.~15.1]{Niren4}. The Schauder estimates holds in the same manner, see \cite[Thm.~7.1, Thm.~7.2]{Niren4}. This is particularly useful since for the global analysis in next section, we will choose a cover of $ (\bar{M}, g) $, and for any boundary chart $ (U, \phi) $ of $(\bar{M}, g) $, the intersection $ \phi(\bar{U} \cap \bar{M}) $ is a one-to-one correspondence to a hemisphere, provided that $ \partial M $ is smooth enough. It resolves the issue for the boundary charts, as we shall see in later sections.
\end{remark}
\medskip

Another tool required in the future analysis is the $ W^{s, p} $-type ``Peter-Paul" inequality.
\begin{proposition}\label{boundary:prop3}\cite[Thm.~7.28]{GT} Let $ (\Omega, g) $ be a Riemannian domain in $ \R^{n} $ and $ u \in W^{2, p}(\Omega, g) $. Then for any $ \gamma > 0 $, 
\begin{equation}\label{boundary:eqn7}
\lVert \nabla_{g} u \rVert_{\calL^{p}(\Omega, g)} \leqslant \gamma \lVert u \rVert_{W^{2, p}(\Omega, g)} + C_{\gamma}' \gamma^{-1} \lVert u \rVert_{\calL^{p}(\Omega, g)}.
\end{equation}
Here $ C_{\gamma} $ only depends on $ \gamma $ and $ (\Omega, g) $.
\end{proposition}
Note that in \cite{GT} this inequality is stated in a more general version, we only need the $ W^{2, p} $-case here. Note also that the result above can be easily obtained by Gagliardo-Nirenberg interpolation inequality when $ u $ is compactly supported in $ \Omega $. We can easily extend this local results to global results.
\begin{proposition}\label{boundary:prop4} Let $ (\bar{M}, g) $ be a compact manifold with smooth boundary $ \partial M $. Let $ u \in W^{2, p}(M, g) $. Then for any $ 0 < \gamma < 1 $,
\begin{equation}\label{boundary:eqn8}
\lVert u \rVert_{W^{1, p}(M, g)} \leqslant \gamma \lVert u \rVert_{W^{2, p}(M, g)} + C_{\gamma} \gamma^{-1} \lVert u \rVert_{\calL^{p}(M, g)}.
\end{equation}
Here $ C_{\gamma} $ only depends on $ \gamma $ and $ (\bar{M}, g) $.
\end{proposition}
\begin{proof}
Taking a finite cover $ (U_{\xi}, \phi_{\xi}) $ of $ (\bar{M}, g) $ and a smooth partition of unity $ \lbrace \chi_{\xi} \rbrace $. Applying this, we have
\begin{align*}
\lVert u \rVert_{W^{1, p}(M, g)} & = \sum_{\xi} \left\lVert \left(\phi_{\xi}^{-1}\right)^{*} \chi_{\xi} u \right\rVert_{W^{1, p}(\phi_{\xi}(U_{\xi}), g)} = \lVert u \rVert_{\calL^{p}(M, g)} + \sum_{\xi} \left\lVert \nabla_{g} \left(\left(\phi_{\xi}^{-1}\right)^{*} \chi_{\xi} u \right) \right\rVert_{\calL^{p}(\phi_{\xi}(U_{\xi}), g)} \\
& \leqslant \lVert u \rVert_{\calL^{p}(M, g)} + \sum_{\xi} \left( \gamma \left\lVert \left(\phi_{\xi}^{-1}\right)^{*} \chi_{\xi} u \right\rVert_{W^{2, p}(\phi_{\xi}(U_{\xi}), g)} + C_{\gamma}' \gamma^{-1} \left\lVert \left(\phi_{\xi}^{-1}\right)^{*} \chi_{\xi}u \right\rVert_{\calL^{p}(\phi_{\xi}(U_{\xi}), g)} \right) \\
& = \gamma \lVert u \rVert_{W^{2, p}(M, g)} + C_{\gamma} \gamma^{-1} \lVert u \rVert_{\calL^{p}(M, g)}.
\end{align*}
In the last step, we combine the first and the third term.
\end{proof}
\medskip

Sobolev embedding theorem for compact manifolds with boundary plays an important role in regularity arguments.
\begin{proposition}\label{boundary:prop5}\cite[Ch.~2]{Aubin}  (Sobolev Embeddings) 
Let $ (\bar{M}, g) $ be a compact manifold with smooth boundary $ \partial M $.

(i) For $ s \in \mathbb{N} $ and $ 1 \leqslant p \leqslant p' < \infty $ such that
\begin{equation}\label{boundary:eqn9}
   \frac{1}{p} - \frac{s}{n} \leqslant \frac{1}{p'},
\end{equation}
\noindent  $ W^{s, p}(M, g) $ continuously embeds into $ \mathcal{L}^{p'}(M, g) $ with the following estimates: 
\begin{equation}\label{boundary:eqn9a}
\lVert u \rVert_{\calL^{p'}(M, g)} \leqslant K \lVert u \rVert_{W^{s, p}(M, g)}.
\end{equation}

(ii) For $ s \in \mathbb{N} $, $ 1 \leqslant p < \infty $ and $ 0 < \alpha < 1 $ such that
\begin{equation}\label{boundary:eqn10}
  \frac{1}{p} - \frac{s}{n} \leqslant -\frac{\alpha}{n},
\end{equation}
Then  $ W^{s, p}(M, g) $ continuously embeds in the H\"older space $ \calC^{0, \alpha}(\bar{M}) $ with the following estimates:
\begin{equation}\label{boundary:eqn10a}
\lVert u \rVert_{\calC^{0, \alpha}(\bar{M})} \leqslant K' \lVert u \rVert_{W^{s, p}(M, g)}.
\end{equation}

(iii) Both embeddings above are compact embeddings provided that the equalities in (\ref{boundary:eqn9}) and (\ref{boundary:eqn10}) do not hold, respectively.
\end{proposition} 
\medskip

In order to deal with manifolds with boundaries, a trace theorem is often required. Let $ \imath : \partial M \rightarrow \bar{M} $ be the inclusion map and thus $ \partial M $ admits an induced Riemannian metric $ \imath^{*} g $. The following version is due to Taylor \cite{T}.
\begin{proposition}\label{boundary:prop6}\cite[Prop.~4.5]{T}
Let $ (\bar{M}, g) $ be a compact manifold with smooth boundary $ \partial M $. Let $ u \in H^{1}(M, g) $. Then there exists a bounded linear operator
\begin{equation*}
T : H^{1}(M, g) \rightarrow \calL^{2}(\partial M, \imath^{*}g)
\end{equation*}
such that
\begin{equation}\label{boundary:eqn11}
\begin{split}
T u & = u \bigg|_{\partial M}, \; \text{if} \; u \in \calC^{\infty}(\bar{M}) \cap H^{1}(M, g); \\
\lVert T u \rVert_{\calL^{2}(\partial M, \imath^{*} g)} & \leqslant K'' \lVert u \rVert_{H^{1}(M, g)}.
\end{split}
\end{equation}
Here $ K'' $ only depends on $ (M, g) $ and is independent of $ u $. Furthermore, the map $ T : H^{1}(M, g) \rightarrow H^{\frac{1}{2}}(\partial M, \imath^{*} g) $ is surjective.
\end{proposition}
\medskip

The following result, which is a local version of perturbed Yamabe equation with trivial Dirichlet boundary condition and a negative constant $ \beta < 0 $, plays a central role in boundary Yamabe problem. We proved this result in \cite{XU3}, and applied this result to proof Yamabe problem on closed manifolds.
\begin{proposition}\label{boundary:prop7}\cite[Prop.~3.3]{XU3}
Let $ (\Omega, g) $ be Riemannian domain in $\R^n$, $ n \geqslant 3 $, with $C^{\infty} $ boundary, and with ${\rm Vol}_g(\Omega)$ and the Euclidean diameter of $\Omega$ sufficiently small. Let $ \beta < 0 $ be any negative constant. Assume $ S_{g} < 0 $ within the small enough closed domain $ \bar{\Omega} $. Then for any $ \lambda > 0 $ the following Dirichlet problem
\begin{equation}\label{boundary:eqn12}
-a\Delta_{g} u + \left( S_{g} + \beta \right) u = \lambda u^{p-1} \; {\rm in} \; \Omega, u = 0 \; {\rm on} \; \partial \Omega.
\end{equation}
has a real, positive solution $ u \in \calC_{0}(\bar{\Omega}) \cap H_{0}^{1}(\Omega, g) $ vanishes at $ \partial \Omega $. 
\end{proposition}
\begin{remark}\label{boundary:re2}
Let $ \lambda_{1} $ be the first nonzero eigenvalue of $ -\Delta_{g} $ on Riemannian domain $ (\Omega, g) $ with Dirichlet boundary condition. Recall that in Proposition 3.3 of \cite{XU3}, the smallness of $ \Omega $ is determined by
\begin{equation}\label{boundary:eqn13}
\sup_{x \in M} \lvert S_{g} \rvert + \lvert \beta \rvert \leqslant a \lambda_{1}, \frac{a}{n} - \left( \frac{n - 2}{2n} + \frac{1}{2} \right) \left( \sup_{x \in M} \lvert S_{g} \rvert + \lvert \beta \rvert \right) \lambda_{1}^{-1} \geqslant 0.
\end{equation}
(\ref{boundary:eqn13}) will be used in Section 6.
\end{remark}
\medskip

\section{Monotone Iteration Scheme on Closed Manifolds with Boundary}
In this section, an $ \calL^{p} $-regularity result on compact manifolds with smooth boundary will be proved first, this global $ \calL^{p} $ regularity will then be used to show the existence of solution of second order elliptic PDE on $ (\bar{M}, g) $ with appropriate boundary conditions by monotone iteration scheme. Throughout the whole section, the existence of solutions, sub-solutions or super-solutions are assumed. Throughout this section, we assume that $ \dim \bar{M} \geqslant 3 $.
\medskip

The first result is a global $ \calL^{p} $-regularity with respect to the elliptic operator and oblique boundary conditions. This proof, essentially, is due to Agmon, Douglis and Nirenberg \cite{Niren4}, although they only proved a local version.
\begin{theorem}\label{global:thm1} Let $ (\bar{M}, g) $ be a compact manifold with smooth boundary $ \partial M $. Let $ \nu $ be the unit outward normal vector along $ \partial M $. Let $ L $ be a uniform second order elliptic operator on $ M $ with smooth coefficients up to $ \partial M $. Let $ f \in \calL^{p}(M, g) $. Let $ u \in H^{1}(M, g) $ be a weak solution of the following boundary value problem
\begin{equation}\label{global:eqn1}
L u = f \; {\rm in} \; M, \frac{\partial u}{\partial \nu} + c(x) u = 0 \; {\rm on} \; \partial M.
\end{equation}
Here $ c \in \calC^{\infty}(M) $. If, in addition, $ u \in \calL^{p}(M, g) $, then $ u \in W^{2, p}(M, g) $ with the following estimates
\begin{equation}\label{global:eqn2}
\lVert u \rVert_{W^{2, p}(M, g)} \leqslant C \left( \lVert Lu \rVert_{\calL^{p}(M, g)} + \lVert u \rVert_{\calL^{p}(M, g)} \right).
\end{equation}
Here $ C $ depends on $ L, p, c $ and the manifold $ (\bar{M}, g) $ and is independent of $ u $.
\end{theorem}
\begin{proof} Choose a finite cover of $ (\bar{M}, g) $, say
\begin{equation*}
(\bar{M}, g) = \left( \bigcup_{\alpha} (U_{\alpha}, \phi_{\alpha}) \right) \cup \left( \bigcup_{\beta} (U_{\beta}, \phi_{\beta}) \right)
\end{equation*}
where $ \lbrace U_{\alpha}, \phi_{\alpha} \rbrace $ are interior charts and $ \lbrace U_{\beta}, \phi_{\beta} \rbrace $ are boundary charts. Choose a partition of unity $ \lbrace \chi_{\alpha}, \chi_{\beta} \rbrace $ subordinate to this cover, where $ \lbrace \chi_{\alpha} \rbrace $ are associated with interior charts and $ \lbrace \chi_{\beta} \rbrace $ are associated with boundary charts. The local expression of the differential operator for interior charts is of the form
\begin{equation*}
L \mapsto \left( \phi_{\alpha}^{-1} \right)^{*} L \phi_{\alpha}^{*} : \calC^{\infty}(\phi_{\alpha}(U_{\alpha})) \rightarrow \calC^{\infty}(\phi_{\alpha}(U_{\alpha})) 
\end{equation*}
which can be extended to Sobolev spaces with appropriate orders. The same expression applies for boundary charts. Denote
\begin{align*}
L_{\alpha} & =  \left( \phi_{\alpha}^{-1} \right)^{*} L \phi_{\alpha}^{*}, L_{\beta} =  \left( \phi_{\beta}^{-1} \right)^{*} L \phi_{\beta}^{*}; \frac{\partial}{\partial \nu'} =  \left( \phi_{\beta}^{-1} \right)^{*} \left( \frac{\partial}{\partial \nu} \right)\phi_{\beta}^{*} \\
\left(\phi_{\alpha}^{-1} \right)^{*} \chi_{\alpha} & = \chi_{\alpha}', \left(\phi_{\beta}^{-1} \right)^{*} \chi_{\beta} = \chi_{\beta}'; \left(\phi_{\beta}^{-1} \right)^{*} c = c_{\beta}'; \\
\left(\phi_{\alpha}^{-1} \right)^{*} u & = u_{\alpha}', \left(\phi_{\beta}^{-1} \right)^{*} u = u_{\beta}', \left(\phi_{\alpha}^{-1} \right)^{*} f = f_{\alpha}', \left(\phi_{\beta}^{-1} \right)^{*} f = f_{\beta}'
\end{align*}
With these notations, the local expressions of our PDE with respect to $ u_{\alpha}', u_{\beta}' $ associated with (\ref{global:eqn1}) in each chart, respectively, are as follows:
\begin{equation}\label{global:eqn3}
\begin{split}
L_{\alpha} \left( \chi_{\alpha}' u_{\alpha}' \right) - [L_{\alpha}, \chi_{\alpha}']u_{\alpha}'  & = \chi_{\alpha}' f_{\alpha}' \; {\rm in} \; \phi_{\alpha}(U_{\alpha}), \chi_{\alpha}' u_{\alpha}'  = 0 \; {\rm on} \;\partial \phi_{\alpha}(U_{\alpha}); \\
L_{\beta} \left( \chi_{\beta}' u_{\beta}' \right) - [L_{\beta}, \chi_{\beta}']u_{\beta}'  & = \chi_{\beta}' f_{\beta}' \; {\rm in} \; \phi_{\beta}(U_{\beta}), \\
\frac{\partial \chi_{\beta}' u_{\beta}' }{\partial \nu'} + c_{\beta}' \chi_{\beta}' u_{\beta}' & = \frac{\partial \chi_{\beta}'}{\partial \nu'} u_{\beta}' \; {\rm on} \; \partial \phi_{\beta}(\bar{U}_{\beta} \cap \bar{M}), \chi_{\beta}' u_{\beta}' = 0 \; {\rm on} \; \partial \phi_{\beta}(U_{\beta}) \backslash \left( \partial \phi_{\beta}(\bar{U}_{\beta} \cap \bar{M}) \right).
\end{split}
\end{equation}
Here $ [L, \chi] $ is a commutator defined as
\begin{equation*}
[L, \chi] u = L(\chi u) - \chi (Lu).
\end{equation*}
Since $ L $ is a second order differential operator, $ [L, \chi] $ is a first order differential operator. Since the existence of solution of (\ref{global:eqn1}) is assumed, we conclude that local PDEs in (\ref{global:eqn3}) are solvable with $ \chi_{\alpha}' u', \chi_{\beta}' u \in \calL^{p} \cap H^{1} $ in associated domains, respectively. The boundary conditions on boundary charts are also Robin condition satisfying Proposition \ref{boundary:prop2}. The following analysis is due to Melrose \cite{RBM1}. For interior chart, we take $ \psi_{\alpha} \in \calC^{\infty}(M) $ such that $ \psi_{\alpha} = 1 $ on $ \text{supp}(\chi_{\alpha}) $, denote $ \psi_{\alpha}' = \left(\phi_{\alpha}^{-1} \right)^{*} \psi_{\alpha} $, thus
\begin{align*}
\lVert \chi_{\alpha}' u_{\alpha}' \rVert_{W^{2, p}(\phi_{\alpha}(U_{\alpha}), g)} & \leqslant C^{*} \left( \lVert L_{\alpha} (\chi_{\alpha}' u_{\alpha}' ) \rVert_{\calL^{p}(\phi_{\alpha}(U_{\alpha}), g)} + \lVert \chi_{\alpha}' u_{\alpha}' \rVert_{\calL^{p}(\phi_{\alpha}(U_{\alpha}), g)} \right) \\
& \leqslant C^{*} \left( \lVert \chi_{\alpha}' f_{\alpha}' \rVert_{\calL^{p}(\phi_{\alpha}(U_{\alpha}), g)} + \left\lVert [L_{\alpha}', \chi_{\alpha}'] \psi_{\alpha}' u_{\alpha}' \right\rVert_{\calL^{p}(\phi_{\alpha}(U_{\alpha}), g)} + \lVert \chi_{\alpha}' u_{\alpha}' \rVert_{\calL^{p}(\phi_{\alpha}(U_{\alpha}), g)} \right) \\
& \leqslant C^{*} \lVert \chi_{\alpha}' f_{\alpha}' \rVert_{\calL^{p}(\phi_{\alpha}(U_{\alpha}), g)} + C^{*} C_{1, \alpha} \lVert \psi_{\alpha}' u_{\alpha}' \rVert_{\calL^{p}(\phi_{\alpha}(U_{\alpha}), g)} \\
& \qquad + C^{*} C_{2, \alpha} \lVert \nabla_{g} (\psi_{\alpha}' u_{\alpha}') \rVert_{\calL^{p}(\phi_{\alpha}(U_{\alpha}), g)} + C^{*} \lVert \chi_{\alpha}' u_{\alpha}' \rVert_{\calL^{p}(\phi_{\alpha}(U_{\alpha}), g)} \\
& \leqslant C^{*} \lVert \chi_{\alpha}' f_{\alpha}' \rVert_{\calL^{p}(\phi_{\alpha}(U_{\alpha}), g)} + C_{0, \alpha} \lVert \psi_{\alpha}' u_{\alpha}' \rVert_{W^{1, p}(\phi_{\alpha}(U_{\alpha}), g)} + C^{*} \lVert \chi_{\alpha}' u_{\alpha}' \rVert_{\calL^{p}(\phi_{\alpha}(U_{\alpha}), g)} \\
\end{align*}
For boundary chart, we take $ \psi_{\beta} $ correspondingly, and have
\begin{align*}
\lVert \chi_{\beta}' u_{\beta}' \rVert_{W^{2, p}(\phi_{\beta}(U_{\beta}), g)} & \leqslant C^{*} \left( \lVert L_{\beta} (\chi_{\beta}' u_{\beta}' ) \rVert_{\calL^{p}(\phi_{\alpha}(U_{\alpha}), g)} + \left\lVert \frac{\partial \chi_{\beta}'}{\partial \nu'} u_{\beta}' \right\rVert_{W^{1, p}(\phi_{\beta}(U_{\beta}), g)} + \lVert \chi_{\beta}' u_{\beta}' \rVert_{\calL^{p}(\phi_{\beta}(U_{\beta}), g)} \right) \\
& \leqslant C^{*} \lVert \chi_{\beta}' f_{\beta}' \rVert_{\calL^{p}(\phi_{\beta}(U_{\beta}), g)} + C^{*} \left\lVert [L_{\beta}', \chi_{\beta}'] \psi_{\beta}' u_{\beta}' \right\rVert_{\calL^{p}(\phi_{\beta}(U_{\beta}), g)} \\
& \qquad + C^{*} \left\lVert \frac{\partial \chi_{\beta}'}{\partial \nu'} \psi_{\beta}' u_{\beta}' \right\rVert_{W^{1, p}(\phi_{\beta}(U_{\beta}), g)} + C^{*} \lVert \chi_{\beta}' u_{\beta}' \rVert_{\calL^{p}(\phi_{\beta}(U_{\beta}), g)}  \\
& \leqslant C^{*} \lVert \chi_{\beta}' f_{\beta}' \rVert_{\calL^{p}(\phi_{\beta}(U_{\beta}), g)} + C^{*}C_{1, \beta} \lVert \psi_{\beta}' u_{\beta}' \rVert_{\calL^{p}(\phi_{\beta}(U_{\beta}), g)} \\
& \qquad + C^{*}C_{2, \beta} \lVert \nabla_{g} (\psi_{\beta}' u_{\beta}') \rVert_{\calL^{p}(\phi_{\beta}(U_{\beta}), g)} + C^{*}C_{3, \beta} \lVert \psi_{\beta}' u_{\beta}' \rVert_{W^{1, p}(\phi_{\beta}(U_{\beta}), g)} \\
& \qquad \qquad + C^{*} \lVert \chi_{\beta}' u_{\beta}' \rVert_{\calL^{p}(\phi_{\beta}(U_{\beta}), g)} \\
& \leqslant  C^{*} \lVert \chi_{\beta}' f_{\beta}' \rVert_{\calL^{p}(\phi_{\beta}(U_{\beta}), g)} + C_{0, \beta} \lVert \psi_{\beta}' u_{\beta}' \rVert_{W^{1, p}(\phi_{\beta}(U_{\beta}), g)}  + C^{*} \lVert \chi_{\beta}' u_{\beta}' \rVert_{\calL^{p}(\phi_{\beta}(U_{\beta}), g)}.
\end{align*}
Note that this estimate on boundary chart is legitimate, due to Remark \ref{boundary:re1}, thanks to \cite{Niren4}. By local estimates, we conclude that $ u \in W^{2, p}(M, g) $, since each $ \chi_{\alpha} u, \chi_{\beta} u $ is in $ W^{2, p} $. Sum them up, we have
\begin{align*}
\lVert u \rVert_{W^{2, p}(M,g)} & = \sum_{\alpha} \lVert \chi_{\alpha}' u_{\alpha}' \rVert_{W^{2, p}(\phi_{\alpha}(U_{\alpha}), g)} + \sum_{\beta} \lVert \chi_{\beta}' u_{\beta}' \rVert_{W^{2, p}(\phi_{\alpha}(U_{\alpha}), g)} \\
& \leqslant \sum_{\alpha} \left( C^{*} \lVert \chi_{\alpha}' f_{\alpha}' \rVert_{\calL^{p}(\phi_{\alpha}(U_{\alpha}), g)} + C_{0, \alpha} \lVert \psi_{\alpha}' u_{\alpha}' \rVert_{W^{1, p}(\phi_{\alpha}(U_{\alpha}), g)} + C^{*} \lVert \chi_{\alpha}' u_{\alpha}' \rVert_{\calL^{p}(\phi_{\alpha}(U_{\alpha}), g)} \right) \\
& \qquad + \sum_{\beta} \left( C^{*} \lVert \chi_{\beta}' f_{\beta}' \rVert_{\calL^{p}(\phi_{\beta}(U_{\beta}), g)} + C_{0, \beta} \lVert \psi_{\beta}' u_{\beta}' \rVert_{W^{1, p}(\phi_{\beta}(U_{\beta}), g)}  + C^{*} \lVert \chi_{\beta}' u_{\beta}' \rVert_{\calL^{p}(\phi_{\beta}(U_{\beta}), g)} \right) \\
& \leqslant C^{*} \lVert f \rVert_{\calL^{p}(M, g)} + C^{*} \lVert u \rVert_{\calL^{p}(M, g)} + C_{1}^{*} \lVert u \rVert_{W^{1, p}(M, g)}.
\end{align*}
The constant $ C_{1}^{*} $ depends in particular on the choice of finite cover, $ \chi_{\alpha}, \chi_{\beta} $ and $ \psi_{\alpha}, \psi_{\beta} $. Applying ``Peter-Paul" inequality, we have
\begin{align*}
\lVert u \rVert_{W^{2, p}(M, g)} & \leqslant C^{*} \lVert f \rVert_{\calL^{p}(M, g)} + C^{*} \lVert u \rVert_{\calL^{p}(M, g)} + C_{1}^{*} \gamma \lVert u \rVert_{W^{2, p}(M, g)} + C_{1}^{*} C_{\gamma} \gamma^{-1} \lVert u \rVert_{\calL^{p}(M, g)} 
\end{align*}
Taking $ \gamma $ small enough so that we can combine $ C_{1}^{*} \gamma \lVert u \rVert_{W^{2, p}(M, g)} $ to the left side of the inequality above. With an appropriate choice of $ C $, which depends on $ \gamma, p, L, c(x) $, $ (\bar{M}, g) $, and the partition of unity, we have
\begin{equation*}
\lVert u \rVert_{W^{2, p}(M, g)} \leqslant C \left( \lVert Lu \rVert_{\calL^{p}(M, g)} + \lVert u \rVert_{\calL^{p}(M, g)} \right).
\end{equation*}
\end{proof}
\medskip

Next we show that the last term $ \lVert u \rVert_{\calL^{p}(M, g)} $ can be removed when $ L $ is an injective operator on $ W^{2, p}(M, g) $. The following argument is an analogy of \cite[\S7, Remark 2]{Niren4}.
\begin{theorem}\label{global:thm2} Let $ (\bar{M}, g) $ be a compact manifold with smooth boundary $ \partial M $. Let $ \nu $ be the unit outward normal vector along $ \partial M $ and $ p > \dim \bar{M} $. Let $ L: \calC^{\infty}(\bar{M}) \rightarrow \calC^{\infty}(\bar{M}) $ be a uniform second order elliptic operator on $ M $ with smooth coefficients up to $ \partial M $ and can be extended to $ L : W^{2, p}(M, g) \rightarrow \calL^{p}(M, g) $. Let $ f \in \calL^{p}(M, g) $. Let $ u \in H^{1}(M, g) $ be a weak solution of the following boundary value problem
\begin{equation}\label{global:eqn4}
L u = f \; {\rm in} \; M, \frac{\partial u}{\partial \nu} + c(x) u = 0 \; {\rm on} \; \partial M.
\end{equation}
Here $ c \in \calC^{\infty}(M) $. Assume also that $ \text{Ker}(L) = \lbrace 0 \rbrace $ associated with this boundary condition. If, in addition, $ u \in \calL^{p}(M, g) $, then $ u \in W^{2, p}(M, g) $ with the following estimates
\begin{equation}\label{global:eqn5}
\lVert u \rVert_{W^{2, p}(M, g)} \leqslant C' \lVert Lu \rVert_{\calL^{p}(M, g)}.
\end{equation}
Here $ C' $ depends on $ L, p, c $ and the manifold $ (\bar{M}, g) $ and is independent of $ u $.
\end{theorem}
\begin{proof}
Due to (\ref{global:eqn2}), it is suffice to show that there exists some constant $ D $ such that
\begin{equation}\label{global:eqn6}
\lVert u \rVert_{\calL^{p}(M, g)} \leqslant D \lVert Lu \rVert_{\calL^{p}(M, g)}
\end{equation}
for all $ u \in W^{2, p}(M, g) $ satisfying the boundary condition in (\ref{global:eqn4}). Since $ \calC^{\infty}(\bar{M}) $ is dense in $ W^{2, p}(M, g) $, we show this by assuming, without loss of generality, $ u \in \calC^{\infty}(\bar{M}) $.
We show this by contradiction. Suppose that (\ref{global:eqn6}) does not hold. Then there exists a sequence $ \lbrace u_{n} \rbrace \subset \calC^{\infty}(\bar{M)} $, normalized with $ \lVert u_{n} \rVert_{\calL^{p}(M, g)} = 1, \forall k \in \mathbb{Z}_{\geqslant 0} $ such that
\begin{equation*}
\lVert u_{n} \rVert_{\calL^{p}(M, g)} \geqslant n \lVert Lu_{n} \rVert_{\calL^{p}(M, g)} \Rightarrow \lVert Lu_{n} \rVert_{\calL^{p}(M, g)} \leqslant \frac{1}{n}, n \in \mathbb{Z}_{\geqslant 0}.
\end{equation*}
It follows from estimate in (\ref{global:eqn2}) that
\begin{equation*}
\lVert u_{n} \rVert_{W^{2, p}(M, g)} \leqslant C \left(\lVert Lu_{n} \rVert_{\calL^{p}(M, g)} + \lVert u_{n} \rVert_{\calL^{p}(M, g)} \right) \leqslant 2C, \forall n \in \mathbb{Z}_{\geqslant 0}.
\end{equation*}
Thus we obtain a sequence $ \lbrace u_{n} \rbrace $ that is uniformly bounded with $ W^{2, p} $-norm. Since $ p > \dim M $, Sobolev embedding in Proposition \ref{boundary:prop5} implies that a subsequence of $ \lbrace u_{n} \rbrace \subset \calC^{1, \alpha}(\bar{M}) $, say $ \lbrace u_{n_{k}} \rbrace \subset  \calC^{1, \alpha}(\bar{M}) $ converges pointwise to some limit $ u $ with the property that
\begin{equation*}
\lim_{n_{k} \rightarrow \infty} Lu_{n_{k}} = Lu = 0, \lim_{n_{k} \rightarrow \infty} u_{n_{k}} = u, \lVert u \rVert_{\calL^{p}(M, g)} = 1.
\end{equation*}
Since $ \text{Ker}(L) $ is trivial, it follows that $ u = 0 $ due to $ Lu = 0 $. Thus $ \lVert u \rVert_{\calL^{p}(M, g)} = 0 $, which contradicts above. Therefore we conclude that (\ref{global:eqn6}) holds. Applying (\ref{global:eqn6}) to (\ref{global:eqn2}), we conclude that
\begin{equation*}
\lVert u \rVert_{W^{2, p}(M, g)} \leqslant C \left(\lVert Lu \rVert_{\calL^{p}(M, g)} + \lVert u \rVert_{\calL^{p}(M, g)} \right) \leqslant C( 1 + D) \lVert Lu \rVert_{\calL^{p}(M, g)} : = C' \lVert Lu \rVert_{\calL^{p}(M, g)}.
\end{equation*}
Here $ C' $ is, in particular, independent of $ u $.
\end{proof}
\medskip

From now on, we consider the special case $ L = -a\Delta_{g} $. Recall that $ a = \frac{4(n - 1)}{n - 2} $. With the help of Theorem \ref{global:thm2}, we can get a result related to the existence of the solution $ -a\Delta_{g}u + f(x, u) = 0 $ on $ (\bar{M}, g) $ with oblique boundary condition $ \frac{\partial u}{\partial \nu} + c(x) u = 0 $ on $ \partial M $ provided the existence of some sub-solution and super-solution of the above PDE. This result is an extension of the old result due to Sattinger \cite{SA}.
\begin{theorem}\label{global:thm3}
Let $ (\bar{M}, g) $ be a compact manifold with smooth boundary $ \partial M $. Let $ \nu $ be the unit outward normal vector along $ \partial M $ and $ p > \dim \bar{M} $. Let $ -a\Delta_{g}: \calC^{\infty}(\bar{M}) \rightarrow \calC^{\infty}(\bar{M}) $ be a uniform second order elliptic operator on $ M $. Let $ F(x, u) : \bar{M} \times \R \rightarrow \R $ be a function smooth in $ x $ and $ \calC^{1} $ in $ u $. Furthermore, assume that there exists a positive constant $ A_{0} $ such that the operator $ -a\Delta_{g} + A $ is injective on $ \calC^{\infty}(\bar{M}) $ for all $ A \geqslant A_{0} $. In addition, assume that the second order linear elliptic PDE
\begin{equation*}
(-a\Delta_{g} + A)u = f \; {\rm in} \; M, \frac{\partial u}{\partial \nu} + c(x) u = 0 \; {\rm on} \; \partial M
\end{equation*}
has a unique weak solution $ u \in H^{1}(M, g) $, for all $ A \geqslant A_{0} $ with $ f \in \calL^{p}(M, g) $ and $ c \in \calC^{\infty}(\bar{M}) $ with $ c(x) > 0 $ on $ \partial M $. Suppose that there exist $ u_{-}, u_{+} \in \calC_{0}(\bar{M}) \cap H^{1}(M, g) $, $ u_{-} \leqslant u_{+} $ such that
\begin{equation}\label{global:eqn7}
\begin{split}
-a\Delta_{g} u_{-} - F(x, u_{-}) & \leqslant 0 \; {\rm in} \; M, \frac{\partial u_{-}}{\partial \nu} + c(x) u_{-} \leqslant 0 \; {\rm on} \; \partial M; \\
-a\Delta_{g} u_{+} - F(x, u_{+}) & \geqslant 0 \; {\rm in} \; M, \frac{\partial u_{+}}{\partial \nu} + c(x) u_{+} \geqslant 0 \; {\rm on} \; \partial M
\end{split}
\end{equation}
holds weakly. Then there exists solution $ u \in W^{2, p}(M, g) $ of
\begin{equation}\label{global:eqn8}
-a\Delta_{g} u - F(x, u) = 0 \; {\rm in} \; M, \frac{\partial u}{\partial \nu} + c(x) u = 0 \; {\rm on} \; \partial M.
\end{equation}
\end{theorem}
\begin{proof} Since $ F(x, u) $ is smooth in both variables, we observe that $ \frac{\partial F}{\partial u} $ is bounded below for all $ x \in \bar{M} $ and $ u \in [\min_{\bar{M}} u_{-}, \max_{\bar{M}} u_{+} ] $. It follows that we can choose $ A > A_{0} $ such that
\begin{equation}\label{global:eqn9}
\frac{\partial F}{\partial u}(x, u) + A > 0, \forall x \in \bar{M}, \forall u \in [\min_{\bar{M}} u_{-}, \max_{\bar{M}} u_{+} ].
\end{equation}
Choose $ u_{0} = u_{+} $. Set
\begin{equation}\label{global:eqn10}
(-a\Delta_{g} + A) u_{1} = F(x, u_{0}) + Au_{0} \; {\rm in} \; M, \frac{\partial u_{1}}{\partial \nu} + c(x) u_{1} = 0 \; {\rm on} \; \partial M.
\end{equation}
By hypotheses in the statement, such $ u_{1} \in H^{1}(M, g) $ does exist. By assumption $ -a\Delta_{g} + A $ is injective, also note that $ u_{0} = u_{+} \in \calC_{0}(\bar{M}) $ and thus $ u_{0} \in \calL^{p}(M, g) $, thus by Theorem \ref{global:thm2}, we conclude that
\begin{equation*}
\lVert u_{1} \rVert_{W^{2, p}(M, g)} \leqslant C' \lVert F(x, u_{0}) + Au_{0} \rVert_{\calL^{p}(M, g)} \Rightarrow u_{1} \in W^{2, p}(M, g) \Rightarrow u_{1} \in \calC^{1, \alpha}(M, g)
\end{equation*}
for some $ \alpha \leqslant 1 - \frac{n}{p} $ by Sobolev embedding. More importantly, $ u_{1} \leqslant u_{0} $. To see this, we have
\begin{align*}
(-a\Delta_{g} + A) u_{0} & \geqslant F(x, u_{0}) + Au_{0} \\
(-a\Delta_{g} + A) u_{1} & = F(x, u_{0}) + Au_{0}.
\end{align*}
Taking subtraction between two formulas above, we have
\begin{equation}\label{global:eqn11}
(-a\Delta_{g} + A) (u_{0} - u_{1}) \geqslant 0, B_{g}(u_{0} - u_{1}) = \frac{\partial (u_{0} - u_{1})}{\partial \nu} + c(x) (u_{0} - u_{1}) \geqslant 0.
\end{equation}
We claim that
\begin{equation}\label{global:eqn11a}
u_{0} \geqslant u_{1} \; {\rm in} \; \bar{M}.
\end{equation}
To see this, we define $ w = \max \lbrace 0, u_{1} - u_{0} \rbrace $. Since $ u_{1} - u_{0} \in H^{1}(M, g) $ so is $ w $. Furthermore $ w \geqslant 0 $. In addition, we have $ (-a\Delta_{g} + A) w \leqslant 0 $. Thus we have
\begin{equation*}
0 \geqslant \int_{M} w (-a\Delta_{g} u + A)w d\omega = \int_{M} \left( a \lvert \nabla_{g} w \rvert^{2} + A w^{2} \right) d\omega + \int_{\partial M} c(x) w^{2} dS \geqslant 0.
\end{equation*}
Thus we must have $ w \equiv 0 $, which follows that $ u_{1} - u_{0} \leqslant 0 $. Hence the claim in (\ref{global:eqn11a}) holds. By a similar argument, we see that $ u_{1} \geqslant u_{-} $. Inductively, we take
\begin{equation}\label{global:eqn12}
(-a\Delta_{g} + A) u_{k} = F(x, u_{k - 1}) + Au_{k - 1} \; {\rm in} \; M, \frac{\partial u_{k}}{\partial \nu} + c(x) u_{k} = 0 \; {\rm on} \; \partial M, k \in \mathbb{N}.
\end{equation}
By the same argument as above, we conclude that $ u_{k} \in W^{2, p}(M, g) $ and hence $ u_{k} \in \calC^{1, \alpha}(M, g) $ for the same choice of $ \alpha $ as above. We show that $ u_{k+ 1} \leqslant u_{k} $ by assuming that $ u_{-} \leqslant u_{k} \leqslant u_{k - 1} \leqslant u_{+} $ inductively. Observe that
\begin{align*}
(-a\Delta_{g} + A) u_{k + 1} & = F(x, u_{k}) + Au_{k}; \\
(-a\Delta_{g} + A) u_{k} & = F(x, u_{k- 1}) + Au_{k - 1}.
\end{align*}
Taking subtraction, we have
\begin{equation*}
(-a\Delta_{g} + A)(u_{k + 1} - u_{k}) = F(x, u_{k }) + Au_{k} - F(x, u_{k-1}) - Au_{k - 1}.
\end{equation*}
Based on the choice of $ A $ in (\ref{global:eqn9}), we observe that $ F(x, u_{k}) + Au_{k} \leqslant F(x, u_{k - 1}) - Au_{k-1} $ due to mean value theorem as well as the inductive assumption that $ u_{-} \leqslant u_{k} \leqslant u_{k - 1} \leqslant u_{+} $. It follows that
\begin{equation}\label{global:eqn13}
(-a\Delta_{g} + A)(u_{k + 1} - u_{k}) \leqslant 0 \Rightarrow u_{k + 1} \leqslant u_{k}, \forall k \in \mathbb{N}.
\end{equation}
By comparing the equations with inductive assumption $ u_{k -1} \geqslant u_{-} $,
\begin{align*}
(-a\Delta_{g} + A)u_{k} & = F(x, u_{k - 1}) + Au_{k - 1}; \\
(-a\Delta_{g} + A)u_{-} & \leqslant F(x, u_{-}) + Au_{-},
\end{align*}
we conclude by the same argument above that
\begin{equation}\label{global:eqn14}
u_{k} \geqslant u_{-}, \forall k \in \mathbb{N}.
\end{equation}
Combining (\ref{global:eqn12}), (\ref{global:eqn13}) and (\ref{global:eqn14}), we conclude that
\begin{equation}\label{global:eqn15}
u_{-} \leqslant \dotso \leqslant u_{k + 1} \leqslant u_{k} \leqslant u_{k - 1} \leqslant u_{k - 2} \leqslant \dotso \leqslant u_{+}, u_{k} \in W^{2, p}(M, g), \forall k \in \mathbb{N}.
\end{equation}
By Sobolev embedding, we conclude that $ u_{k} \in \calC^{1, \alpha}(M, g) $ for all $ k \in \mathbb{N} $. Furthermore, we observe from (\ref{global:eqn5}) that
\begin{equation*}
\lVert u_{k} \rVert_{W^{2, p}(M, g)} \leqslant C' \lVert (-a\Delta_{g} + A) u_{k - 1} \rVert_{\calL^{p}(M, g)}.
\end{equation*}
Since $ u_{-} \leqslant u_{k} \leqslant u_{+}, \forall k \in \mathbb{N} $, $ \lVert (-a\Delta_{g} + A) u_{k - 1} \rVert_{\calL^{p}(M, g)} $ has a uniform upper bound, and thus $ \lVert u_{k} \rVert $ is uniformly bounded in $ W^{2, p} $-norm. When $ \alpha < 1 - \frac{p}{n} $, the embedding $ \calC^{1, \alpha}(\bar{M}) \hookrightarrow W^{2, p}(M, g) $ is a compact embedding and thus a subsequence of $ u_{k} $ converge to some limit $ u $. According to the chain of inequalities in (\ref{global:eqn15}), the whole sequence $ u_{k} $ converges to $ u $ in $ \calC^{1, \alpha}(\bar{M}) $ and hence in $ W^{2, p} $-sense. Taking the limit, we have
\begin{equation*}
-a\Delta_{g} u - F(x, u) = 0 \; {\rm in} \; M, \frac{\partial u}{\partial \nu} + c(x) u = 0 \; {\rm on} \; \partial M.
\end{equation*}
The boundary condition is achieved in the trace sense. Finally local Schauder estimates \cite[Thm.~7.2,Thm.~7.3]{Niren4} indicates that $ u \in \calC^{2, \alpha}(\bar{M}) $ since regularity is a local property.
\end{proof}
\medskip

\section{Solvability of $ -\Lambda \Delta_{g} u + \Lambda' u = f $ with Oblique Boundary Condition}
We expect to apply results in \S3 to solve boundary Yamabe problem. In order to apply Theorem \ref{global:thm3}, we need to show the existence of the weak solution of the following PDE
\begin{equation}\label{solve:eqn1}
-\Lambda \Delta_{g} u + \Lambda' u = f \; {\rm in} \; M,Bu : = \frac{\partial u}{\partial \nu} + c(x) u = 0 \; {\rm on} \; \partial M
\end{equation}
with appropriate choices of $ \Lambda, \Lambda' $, provided that $ f \in \calL^{2}(M, g) $ and $ c \in \calC^{\infty}(\bar{M}) $. A standard Lax-Milgram \cite[Ch.~6]{Lax} will be applied to verify the solvability of (\ref{solve:eqn1}). Throughout this section, we denote $ \imath : \partial M \rightarrow \bar{M} $ to be the standard inclusion map. The first result below does not require the positivity of $ c $ on $ \partial M $.
\begin{theorem}\label{solve:thm1}
Let $ (\bar{M}, g) $ be a compact manifold with smooth boundary $ \partial M $. Let $ \nu $ be the unit outward normal vector along $ \partial M $. Let $ f \in \calL^{2}(M, g) $ and $ c \in \calC^{\infty}(\bar{M}) $. Then for large enough $ \Lambda, \Lambda' $, (\ref{solve:eqn1}) has a unique weak solution $ u \in H^{1}(M, g) $. Furthermore, the operator $ -\Lambda \Delta_{g} + \Lambda' $ is injective, i.e. $ \text{Ker}(- \Lambda\Delta_{g} + \Lambda') = \lbrace 0 \rbrace $ with respect to the oblique boundary condition in (\ref{solve:eqn1}).
\end{theorem}
\begin{proof} We consider the Lax-Milgram with respect to $ H^{1}(M,g) $. Note that the boundary condition is defined in the weak sense by pairing $ H^{-\frac{1}{2}}(\partial M, \imath^{*}g) \times H^{\frac{1}{2}}(\partial M, \imath^{*} g) $ with the divergence theorem and (\ref{solve:eqn1})
\begin{align*}
\langle Bu, w \rangle & = \int_{\partial M} \left(\frac{\partial u}{\partial \nu} + c(x)u \right) v dS = \int_{M} \left( \nabla_{g} u \cdot \nabla_{g} v + \left(\Delta_{g} u \right)v  \right)d\omega + \int_{\partial M} c uv dS \\
& = \int_{M} \left( \nabla_{g} u \cdot \nabla_{g} v + \frac{\Lambda'}{\Lambda} u v - \frac{1}{\Lambda} fv \right) d\omega + \int_{\partial M} cuv dS.
\end{align*}
Due to the trace theorem in Proposition \ref{boundary:prop6}, we have
\begin{equation}\label{solve:eqn3}
\int_{\partial M} c(x) u^{2} dS \leqslant \sup_{\partial M} \lvert c \rvert \int_{\partial M} u^{2} dS = \sup_{\partial M} \lvert c \rvert \lVert u \rVert_{\calL^{2}(\partial M, \imath^{*} g)}^{2} \leqslant \sup_{\partial M} \lvert c \rvert (K'')^{2} \lVert u \rVert_{H^{1}(M, g)}.
\end{equation}
Note that the constant $ \sup_{\partial M} \lvert c \rvert K'' $ only depends on $ c $ and $ (\bar{M}, g) $ and is independent of $ u $. Thus we choose $ \Lambda, \Lambda' $ such that
\begin{equation}\label{solve:eqn4}
\Lambda > \sup_{\partial M} \lvert c \rvert (K'')^{2} + 1, \Lambda' > \sup_{\partial M} \lvert c \rvert (K'')^{2} + 1.
\end{equation}
We observe that the bilinear form of (\ref{solve:eqn1}) is
\begin{equation}\label{solve:eqn5}
B[u, v] = \int_{M} \left( \Lambda \nabla_{g} u \cdot \nabla_{g} v + \Lambda' u v \right) d\omega + \int_{\partial M} cuv dS, \forall v \in H^{1}(M, g).
\end{equation}
We have
\begin{align*}
\lvert B[u, v] \rvert & \leqslant \Lambda \lVert \nabla_{g} u \rVert_{\calL^{2}(M, g)} \lVert \nabla_{g} v \rVert_{\calL^{2}(M, g)} + \Lambda' \lVert u \rVert_{\calL^{2}(M, g)} \lVert v \rVert_{\calL^{2}(M, g)} \\
& \qquad + \sup_{\partial M} \lvert c \rvert \lVert u \rVert_{\calL^{2}(\partial M, \imath^{*}g)} \lVert v \rVert_{\calL^{2}(\partial M, \imath^{*}g)} \\
& \leqslant D_{1} \lVert u \rVert_{H^{1}(M, g)} \lVert v \rVert_{H^{1}(M, g)}
\end{align*}
for some constant $ D_{1} $. On the other hand we apply (\ref{solve:eqn4}),
\begin{align*}
B[u, u] & = \int_{M} \Lambda \lvert \nabla_{g} u \rvert^{2} d\omega + \int_{M} \Lambda' \lvert u \rvert^{2} d\omega + \int_{\partial M} c(x) u^{2} dS \\
& \geqslant \Lambda \lVert \nabla_{g} u \rVert_{\calL^{2}(M, g)}^{2} + \Lambda' \lVert u \rVert_{\calL^{2}(M, g)}^{2} - \sup_{\partial M} \lvert c \rvert \lVert u \rVert_{\calL^{2}(\partial M, \imath^{*}g)}^{2} \\
& \geqslant \Lambda \lVert \nabla_{g} u \rVert_{\calL^{2}(M, g)}^{2} + \Lambda' \lVert u \rVert_{\calL^{2}(M, g)}^{2} - \sup_{\partial M} \lvert c \rvert (K'')^{2} \lVert u \rVert_{H^{1}(M, g)}^{2} \\
& = \left( \Lambda -  \sup_{\partial M} \lvert c \rvert (K'')^{2} \right) \lVert \nabla_{g} u \rVert_{\calL^{2}(M, g)} + \left(\Lambda' -  \sup_{\partial M} \lvert c \rvert (K'')^{2} \right) \lVert u \rVert_{\calL^{2}(M, g)}^{2} \\
& \geqslant \lVert u \rVert_{H^{1}(M, g)}^{2}.
\end{align*}
Hence the hypotheses for Lax-Milgram theorem satisfied. Applying Lax-Milgram theorem, we conclude that there exists some $ u \in H^{1}(M, g) $ that solves (\ref{solve:eqn1}) weakly.
\medskip

If we have
\begin{equation*}
-\Lambda \Delta_{g} u + \Lambda' u = 0 \; {\rm in} \; M, Bu = 0 \; {\rm on} \; \partial M,
\end{equation*}
we pair both sides with $ u $, a very similar argument as above implies that
\begin{align*}
& \int_{M} \left( -\Lambda \Delta_{g} u + \Lambda' u \right) u = 0 \Rightarrow \Lambda \lVert \nabla_{g} u \rVert_{\calL^{2}(M, g)} + \Lambda' \lVert u \rVert_{\calL^{2}(M, g)} + \int_{\partial M} c u^{2} dS = 0 \\
\Rightarrow & 0 \geqslant  \left( \Lambda -  \sup_{\partial M} \lvert c \rvert (K'')^{2} \right) \lVert \nabla_{g} u \rVert_{\calL^{2}(M, g)} + \left(\Lambda' -  \sup_{\partial M} \lvert c \rvert (K'')^{2} \right) \lVert u \rVert_{\calL^{2}(M, g)}^{2} \geqslant \lVert u \rVert_{H^{1}(M, g)}^{2} \\
\Rightarrow & u \equiv 0.
\end{align*}
Hence the operator $ -\Lambda \Delta_{g} u + \Lambda' u : \calC^{\infty}(\bar{M}) \rightarrow \calC^{\infty}(\bar{M}) $ is injective. The extension of this operator to any positive integer order Sobolev spaces is still injective. Note that the solvability of (\ref{solve:eqn1}) and the injectivity of the operator still hold for any larger $ \Lambda, \Lambda' $.
\end{proof}
\medskip

If we further assume $ c > 0 $ on $ \partial M $, then the above result holds for any positive $ \Lambda $ and any nonnegative $ \Lambda' $. It turns to be a global extension of a local result in Gilbarg and Trudinger \cite[Thm.~6.31]{GT}. With the condition $ c > 0 $ on $ \partial M $, the following result plays a key role in applying Theorem \ref{global:thm3}.
\begin{theorem}\label{solve:thm2}
Let $ (\bar{M}, g) $ be a compact manifold with smooth boundary $ \partial M $. Let $ \nu $ be the unit outward normal vector along $ \partial M $. Let $ a, A > 0 $ be any constants. Let $ f \in \calL^{2}(M, g) $ and $ c \in \calC^{\infty}(\bar{M}) $ with $ c > 0 $ on $ \partial M $. Then the following PDE
\begin{equation}\label{solve:eqn6}
-a\Delta_{g} u + Au = f \; {\rm in} \; M, \frac{\partial u}{\partial \nu} + c(x) u = 0 \; {\rm on} \; \partial M
\end{equation}
has a unique weak solution $ u \in H^{1}(M, g) $. Furthermore, $ -a\Delta_{g} + A $ is injective with this oblique boundary condition.
\end{theorem}
\begin{proof} The bilinear form, similar to (\ref{solve:eqn5}), is
\begin{equation}\label{solve:eqn7}
B[u, v] = \int_{M} \left( a \nabla_{g} u \cdot \nabla_{g} v + A u v \right) d\omega + \int_{\partial M} cuv dS, \forall v \in H^{1}(M, g).
\end{equation}
The upper bound of $ \lvert B[u, v] \rvert $ is exactly the same as in Theorem \ref{solve:thm1} with different constant. The lower bound of $ B[u, u] $ is
\begin{equation*}
B[u, u] = a \lVert \nabla_{g} u \rVert_{\calL^{2}(M, g)}^{2} + A \lVert u \rVert_{\calL^{2}(M, g)} + \int_{\partial M} c u^{2} dS \geqslant \min (a, A) \lVert u \rVert_{H^{1}(M, g)}^{2}.
\end{equation*}
Since $ \min (a, A) > 0 $, it follows from Lax-Milgram that (\ref{solve:eqn6}) has a unique weak solution $ u \in H^{1}(M, g) $. For injectivity, we observe that if $ B[u, u] = 0 $ then $ \lVert u \rVert_{H^{1}(M, g)} \leqslant 0 $ hence $ u \equiv 0 $.
\end{proof}

\section{Boundary Yamabe Problem with Minimal Boundary Case}
Recall the boundary value problem associated with boundary Yamabe problem for minimal boundary case.
\begin{equation}\label{yamabe:eqn1}
\begin{split}
\Box_{g} u & : =  -a\Delta_{g} u + S_{g} u = \lambda u^{p-1} \; {\rm in} \; M; \\
B_{g} u & : =  \frac{\partial u}{\partial \nu} + \frac{2}{p-2} h_{g} u = 0 \; {\rm on} \; \partial M.
\end{split}
\end{equation}
In this section, We would apply the sub-solution and super-solution technique in Theorem \ref{global:thm3}  to solve boundary Yamabe equation for five cases:
\begin{enumerate}[(A).]
\item $ \eta_{1} = 0 $;
\item $ \eta_{1} < 0 $ with $ h_{g} > 0 $ everywhere on $ \partial M $ and arbitrary $ S_{g} $;
\item $ \eta_{1} < 0 $ with arbitrary $ h_{g} $ and $ S_{g} $;
\item $ \eta_{1} > 0 $ with $ h_{g} > 0 $ everywhere and $ S_{g} < 0 $ somewhere;
\item $ \eta_{1} > 0 $ with arbitrary $ h_{g} $ and $ S_{g} $.
\end{enumerate}
Throughout this section, we assume $ \dim \bar{M} \geqslant 3 $. We always assume that $ (\bar{M}, g) $ be a compact manifold with smooth boundary $ \partial M $ and $ \nu $ be the unit outward normal vector along $ \partial M $. Let $ S_{g} $ be the scalar curvature and $ h_{g} $ be the mean curvature on $ \partial M $. Let $ \tilde{S} $ and $ \tilde{h} $ be the scalar and mean curvature with respect to $ \tilde{g} $ under conformal change. Note that case (B) is a special scenario of (C), and case (D) is a special scenario of case (E). As in closed manifold case \cite{XU3}, we need an extra step when $ \eta_{1} > 0 $: we need to solve the perturbed boundary Yamabe equation with $ \beta < 0 $
\begin{equation}\label{yamabe:eqns1}
-a\Delta_{g} u_{\beta} + \left(S_{g} + \beta \right) u_{\beta} = \lambda_{\beta} u_{\beta}^{p-1} \; {\rm in} \; M, \frac{\partial u_{\beta}}{\partial \nu} + \frac{2}{p - 2} h_{g} u_{\beta} = 0 \; {\rm on} \; \partial M
\end{equation}
where $ \lambda_{\beta} $ is defined to be
\begin{equation}\label{yamabe:eqns2}
\lambda_{\beta} = \inf_{u \neq 0}  \frac{\int_{M} a \lvert \nabla_{g} u \rvert^{2} d\omega + \int_{M} \left( S_{g} + \beta \right) u^{2} d\omega + \int_{\partial M} \frac{2a}{p-2}h_{g} u^{2} dS}{\left( \int_{M} u^{p} d\omega \right)^{\frac{2}{p}} }.
\end{equation}
Then we take the limit by letting $ \beta \rightarrow 0^{-} $ to obtain the solution of boundary Yamabe equation.
\medskip

We need the following results for the eigenvalue problem with respect conformal Laplacian $ \Box_{g} $ and Robin condition.
\begin{theorem}\label{yamabe:thm1}\cite[Lemma~1.1]{ESC}
Let $ (\bar{M}, g) $ be a compact manifold with boundary. The following eigenvalue problem 
\begin{equation}\label{yamabe:eqn2}
-a\Delta_{g} \varphi + S_{g} \varphi = \eta_{1} \varphi \; {\rm in} \; M, \frac{\partial \varphi}{\partial \nu} + \frac{2}{p - 2} h_{g} \varphi = 0 \; {\rm on} \; \partial M.
\end{equation}
has a real, smooth, positive solution $ \varphi \in \calC^{\infty}(\bar{M}) $.
\end{theorem}
\medskip

When $ \eta_{1} = 0 $ in case (A), we can solve Yamabe problem trivially with $ \lambda = 0 $, this is just an eigenvalue problem. It is worth mentioning that generically zero is not an eigenvalue of conformal Laplacian $ \Box $, see \cite{GHJL}. 
\begin{corollary}\label{yamabe:cor1}
Let $ (\bar{M}, g) $ be a compact manifold with boundary and $ \eta_{1} = 0 $. Then the boundary Yamabe equation (\ref{yamabe:eqn1}) has a real, positive, smooth solution with $ \lambda = 0 $.
\end{corollary}
\begin{proof} It is an immediate consequence of Theorem \ref{yamabe:thm1}.
\end{proof}
\medskip

Now we consider the case when $ \eta_{1} <  0 $, $ h_{g} > 0 $ everywhere on $ \partial M $. The sign of $ S_{g} $ is not required.
\begin{theorem}\label{yamabe:thm2}
Let $ (\bar{M}, g) $ be a compact manifold with boundary. Let $ h_{g} > 0 $ everywhere on $ \partial M $. When $ \eta_{1} < 0 $, there exists some $ \lambda < 0 $ such that the boundary Yamabe equation (\ref{yamabe:eqn1}) has a real, positive solution $ u \in \calC^{\infty}(M) $.
\end{theorem}
\begin{proof} We apply Theorem \ref{global:thm3} to construct sub-solution and super-solution here. By Theorem \ref{solve:thm2}, any $ A_{0} > 0 $ in hypotheses of Theorem \ref{global:thm3} works for the solvability of the linear iterations in (\ref{global:eqn12}).

We construct the sub-solution first and make choice of $ \lambda $. By Theorem \ref{yamabe:thm1}, there exists some $ \varphi $ such that
\begin{equation*}
-a\Delta_{g} \varphi + S_{g} \varphi = \eta_{1} \varphi \; {\rm in} \; M, \frac{\partial \varphi}{\partial \nu} + \frac{2}{p - 2} h_{g} \varphi = 0 \; {\rm on} \; \partial M.
\end{equation*}
Scaling $ \varphi $ by $ t \varphi $, we may, without loss of generality, assume that $ \sup_{\bar{M}} \varphi < 1 $. Since $ p - 1 > 1 $, we have $ \varphi^{p-1} \leqslant \varphi $. Since $ \eta_{1} < 0 $, we can choose $ \lambda \in (\eta_{1}, 0) $ and have
\begin{equation}\label{yamabe:eqn3}
-a\Delta_{g} \varphi + S_{g} \varphi = \eta_{1} \varphi \leqslant \eta_{1} \varphi^{p-1} \leqslant \lambda \varphi^{p-1}.
\end{equation}
Define
\begin{equation}\label{yamabe:eqn4}
F(x, u) : = - S_{g} u + \lambda u^{p-1} : \bar{M} \rightarrow \R.
\end{equation}
It is clear that $ F(x, u) $ is smooth in $ x $ and $ \calC^{1} $ in $ u $. Set
\begin{equation*}
u_{-} : = \varphi > 0, u_{-} \in \calC^{\infty}(\bar{M}).
\end{equation*}
It follows that $ u_{-} $ satisfies
\begin{equation}\label{yamabe:eqn5}
-a \Delta_{g} u_{-} \leqslant F(x, u_{-}) \; {\rm in} \; M, B_{g} u_{-} \leqslant 0 \; {\rm on} \; \partial M.
\end{equation}

Next we construct the super-solution. Choose a constant $ K_{1} > 0 $ that is large enough so that for the choice of $ \lambda < 0 $ in (\ref{yamabe:eqn3}),
\begin{equation*}
K_{1}^{p-2} \geqslant \max \left\lbrace \frac{\inf_{\bar{M}} S_{g}}{\lambda}, \sup_{\bar{M}} u_{-}^{p-2} \right\rbrace.
\end{equation*}
Note that if $ S_{g} \geqslant 0 $ everywhere then the first quantity on the right side above is nonpositive. Set
\begin{equation*}
u_{+} : = K_{1}, u_{+} \in \calC^{\infty}(\bar{M}).
\end{equation*}
Since we assume that $ h_{g} \geqslant 0 $, we check that
\begin{align*}
& -a\Delta_{g} u_{+} + S_{g} u_{+} - \lambda u_{+}^{p-1} = K_{1}( S_{g} - \lambda K_{1}^{p-2}) \geqslant 0 \; {\rm in} \; M; \\
& B_{g} u_{+} = \frac{\partial u_{+}}{\partial \nu} + \frac{2}{p-2} h_{g} u_{+} = \frac{2}{p-2} h_{g} K_{1} \geqslant 0 \; {\rm on} \; \partial M.
\end{align*}
Hence
\begin{equation}\label{yamabe:eqn6}
-a \Delta_{g} u_{+} \geqslant F(x, u_{+}) \; {\rm in} \; M, B_{g} u_{+} \geqslant 0 \; {\rm on} \; \partial M.
\end{equation}
In order to apply Theorem \ref{global:thm3}, we need to choose $ A $ such that (\ref{global:eqn9}) is satisfied. In addition, the choice of $ A $ must guarantee the solvability of (\ref{global:eqn10}). Here we choose $ A $ large enough such that
\begin{equation}\label{yamabe:eqn7}
A +  \sup_{x \in \bar{M}, u \in \left[\min_{\bar{M}} u_{-}, \max_{\bar{M}} u_{+}\right] } \frac{\partial F(x, u)}{\partial u} > 0, A > 0.
\end{equation}
Here $ F $ is of the expression in (\ref{yamabe:eqn4}). Note that with this choice of $ A $, the operator $ -a \Delta_{g} u + A $ is injective. Thus all hypotheses in Theorem \ref{global:thm3} are satisfied. By (\ref{yamabe:eqn5}), (\ref{yamabe:eqn6}) and (\ref{yamabe:eqn7}), It follows that there exists $ u \in W^{2, p}(M, g), p > \dim \bar{M} $ and $ 0 < u_{-} \leqslant u \leqslant u_{+} $ such that
\begin{equation*}
-a \Delta_{g} u = F(x, u) \; {\rm in} \; M, B_{g} u = 0 \; {\rm on} \; \partial M.
\end{equation*}
It is immediate that
\begin{equation*}
- a \Delta_{g} u = F(x, u) = - S_{g} u + \lambda u^{p-1} \Rightarrow -a\Delta_{g} u + S_{g} u = \lambda u^{p-1}.
\end{equation*}
Furthermore, a standard bootstrapping argument with Sobolev embedding and Schauder estimates implies that $ u \in \calC^{\infty}(\bar{M}) $, see e.g. \cite[Thm.~2.8]{XU2} since the regularity is a local argument.
\end{proof}
\medskip

It is not always possible that $ h_{g} \geqslant 0 $ everywhere on $ \partial M $. To handle the general case for arbitrary $ h_{g}, S_{g} $ with $ \eta_{1} < 0 $, we show that there exists some conformal metric $ \tilde{g} $ such that $ \tilde{h} \geqslant 0 $ everywhere. The next result works for both signs of $ \eta_{1} $. Note that by Proposition \ref{boundary:prop1}, $ \text{sgn}\tilde{\eta}_{1} = \text{sgn}(\eta_{1}) $ after conformal change.
\begin{theorem}\label{yamabe:thm3}
Let $ (\bar{M}, g) $ be a compact manifold with boundary. There exists a conformal metric $ \tilde{g} $ associated with mean curvature $ \tilde{h} > 0 $ everywhere on $ \partial M $.
\end{theorem}
\begin{proof} By scaling the metric $ g $, we can, without loss of generality, assume that $ \frac{2}{p - 2} \sup_{\partial M} \lvert h_{g} \rvert \leqslant 1 $. We start with this metric. Pick up a positive function $ H \in \calC^{\infty}(\partial M) $ such that $ \inf_{\partial M} H \geqslant 2 $. We can always find a smooth function $ W \in \calC^{\infty}(\bar{M}) $ such that 
\begin{equation*}
\frac{\partial W}{\partial \nu} = H \; {\rm on} \; \partial M.
\end{equation*}
This $ W $ can be obtained by the unique smooth solution of the PDE
\begin{equation*}
-\Delta_{g} W + W = 0 \; {\rm in} \; M, \frac{\partial W}{\partial \nu} = H \; {\rm on} \; \partial M.
\end{equation*}
Let $ \tilde{W}(x) = e^{x} $ as a map between $ \R $. From the following composition
\begin{equation*}
\begin{tikzcd}
\bar{M} \arrow[r, "W"] & \R \arrow[r, "\tilde{W}"] & \R
\end{tikzcd}
\end{equation*}
We define
\begin{equation}\label{yamabe:eqn8}
u = \tilde{W} \circ W = e^{W} : \bar{M} \rightarrow \R.
\end{equation}
Clearly $ u > 0 $ and is smooth on $ \bar{M} $; in addition $ u \bigg|_{\partial M} $ has the same expression by Proposition \ref{boundary:prop6}. By (\ref{yamabe:eqn8}), we define
\begin{equation}\label{yamabe:eqn9}
\begin{split}
\tilde{h} & = \frac{p - 2}{2} u^{\frac{2}{p}} \left( \frac{\partial u}{\partial \nu} + \frac{2}{p - 2} h_{g} u \right) \; {\rm on} \; \partial M; \\
\tilde{S} & = u^{1 - p} \left( -a\Delta_{g} u + S_{g} u \right) \; {\rm in} \; M.
\end{split}
\end{equation}
Definitions in (\ref{yamabe:eqn9}) implies that the following boundary value problem
\begin{equation*}
\begin{split}
& -a\Delta_{g} u + S_{g} u = \tilde{S} u^{p-1} \; {\rm in} \; M; \\
& \frac{\partial u}{\partial \nu} = \frac{2}{p-2} \left( - h_{g} u + \tilde{h} u^{\frac{p}{2}} \right) \; {\rm on} \; \partial M
\end{split}
\end{equation*}
has a real, smooth, positive solution, i.e. there exists a conformal change $ \tilde{g} = u^{p-2} g $ associated with scalar curvature $ \tilde{S} $ and mean curvature $ \tilde{h} $.

We check the sign of $ \tilde{h} $. By (\ref{yamabe:eqn9}), 
\begin{align*}
\tilde{h} & = \frac{p - 2}{2} u^{\frac{2}{p}} \left( \frac{\partial u}{\partial \nu} + \frac{2}{p - 2} h_{g} u \right) = \frac{p - 2}{2} u^{\frac{2}{p}}  \left(e^{W} \frac{\partial W}{\partial \nu} \bigg|_{\partial M} + \frac{2}{p - 2} h_{g} e^{W} \right) \\
& =  \frac{p - 2}{2} u^{\frac{2}{p}}e^{W} \left( H + \frac{2}{p - 2} h_{g} \right) \geqslant  \frac{p - 2}{2} u^{\frac{2}{p}}e^{W} \left( \inf_{\partial M} H - \frac{2}{p - 2} \sup_{\partial M} \lvert h_{g} \rvert \right) > 0.
\end{align*}
\end{proof}
\medskip

Similarly, we can find a conformal change with $ \tilde{h} < 0 $ everywhere on $ \partial M $, by exactly the same argument as above.
\begin{corollary}\label{yamabe:cor2}
Let $ (\bar{M}, g) $ be a compact manifold with boundary. There exists a conformal metric $ \tilde{g} $ associated with mean curvature $ \tilde{h} < 0 $ everywhere on $ \partial M $.
\end{corollary}
\begin{proof} After possible scaling we start with a metric with $ \frac{2}{p - 2} \sup_{\partial M} \lvert h_{g} \rvert \leqslant 1 $. Now we choose a negative smooth function $ H \in \calC^{\infty}(\partial M) $ such that $ \sup_{\partial M} H \leqslant - 2 $. The rest follows exactly the same as in Theorem \ref{yamabe:thm3}.
\end{proof}
\medskip

With the help of Theorem \ref{yamabe:thm3}, we can deal with the general case when $ \eta_{1} $ without restriction of $ h_{g} $ and $ S_{g} $.
\begin{theorem}\label{yamabe:thm4}
Let $ (\bar{M}, g) $ be a compact manifold with boundary. When $ \eta_{1} < 0 $, there exists some $ \lambda < 0 $ such that the boundary Yamabe equation (\ref{yamabe:eqn1}) has a real, positive solution $ u \in \calC^{\infty}(M) $.
\end{theorem}
\begin{proof} If $ h_{g} > 0 $ everywhere on $ \partial M $, it reduces to the case in Theorem \ref{yamabe:thm2}. If not, then applying Theorem \ref{yamabe:thm3}, we have a conformal metric $ \tilde{g}_{1} = u^{p-2} g $ associate with scalar curvature $ \tilde{S}_{1} $ and  mean curvature $ \tilde{h}_{1} > 0 $ everywhere on $ \partial M $. Here $ u \in \calC^{\infty}(\bar{M}) $ is real and positive. By Theorem \ref{yamabe:thm2}, it follows that there exists a real and positive function $ v \in \calC^{\infty}(\bar{M}) $ such that the metric $ \tilde{g} = v^{p-2} \tilde{g}_{1} $ admits a constant scalar curvature $ \lambda < 0 $ and the boundary is minimal with respect to $ \tilde{g} $. Hence
\begin{equation*}
\tilde{g} = v^{p-2} \tilde{g}_{1} = v^{p-2} u^{p-2} g = (uv)^{p-2} g.
\end{equation*}
\end{proof}
\medskip

When the first eigenvalue $ \eta_{1} > 0 $, we need to control not only the sign of $ h_{g} $ on $ \partial M $, but also the sign of $ S_{g} $. We show the special case when $ h_{g} > 0 $ everywhere on $ \partial M $ and $ S_{g} < 0 $ somewhere, provided that $ \eta_{1} > 0 $. The following argument is inspired by Theorem 4.3 in \cite{XU3}. We start with the perturbed boundary Yamabe equation (\ref{yamabe:eqns1}) first. It is clear that when $ \eta_{1} > 0 $, the Yamabe invariant
\begin{equation*}
Q(M) = \inf_{u \neq 0}  \frac{\int_{M} \left( a \lvert \nabla_{g} u \rvert^{2} + S_{g} u^{2} \right) d\omega + \int_{\partial M} \frac{2a}{p-2}h_{g} u^{2} dS}{\left( \int_{M} u^{p} \dvol \right)^{\frac{2}{p}}}
\end{equation*}
introduced at the beginning is also positive. Hence by the same argument as in Lemma 4.1 of \cite{XU3}, we conclude that $ \lambda_{\beta} > 0 $ when $ \beta < 0 $ with small enough $ \lvert \beta \rvert $.
\begin{theorem}\label{yamabe:thm5}
Let $ (\bar{M}, g) $ be a compact manifold with boundary. Let $ \beta < 0 $ be a fixed constant with small enough constant. Let $ \lambda_{\beta} $ be given in (\ref{yamabe:eqns2}) for the fixed $ \beta $. Assume $ S_{g} < 0 $ somewhere in $ M $ and $ h_{g} > 0 $ everywhere on $ \partial M $. When $ \eta_{1} > 0 $, the perturbed boundary Yamabe equation (\ref{yamabe:eqns1}) has a real, positive solution $ u \in \calC^{\infty}(M) $.
\end{theorem}
\begin{proof} Again we apply Theorem \ref{global:thm3} to construct sub-solution and super-solution here. By Theorem \ref{solve:thm2}, any $ A_{0} > 0 $ in hypotheses of Theorem \ref{global:thm3} works for the solvability of the linear iterations in (\ref{global:eqn12}). When $ \eta_{1} > 0 $ and $ \lvert \beta \rvert $ small enough, we have $ \lambda_{\beta} > 0 $.

According to the eigenvalue problem in Theorem \ref{yamabe:thm1}, the following PDE
\begin{equation*}
-a\Delta_{g} \varphi + \left(S_{g} + \beta \right) \varphi = \eta_{1} \varphi + \beta \varphi \; {\rm in} \; M, \frac{\partial \varphi}{\partial \nu} + \frac{2}{p - 2} h_{g} \varphi = 0 \; {\rm on} \; \partial M
\end{equation*}
has a real, positive solution $ \varphi \in \calC^{\infty}(\bar{M}) $. Note that any scaling $ \delta \varphi $ also solves the PDE above. For the given $ \lambda_{\beta} $, we want
\begin{equation*}
\left( \eta_{1} + \beta \right) \inf_{M} (\delta \varphi) > 2^{p-2} \lambda_{\beta} \sup_{M} \left(\delta^{p-1} \varphi^{p-1} \right) \Leftrightarrow \frac{\left( \eta_{1} + \beta \right)}{2^{p-2}\lambda_{\beta}} > \delta^{p-2} \frac{\sup_{M} \varphi^{p-1}}{\inf_{M} \varphi}.
\end{equation*}
For fixed $ \eta_{1}, \lambda_{\beta}, \varphi, \beta $, this can be done by letting $ \delta $ small enough. We denote $ \phi = \delta \varphi $. It follows that
\begin{equation}\label{yamabe:eqns3}
\begin{split}
-a\Delta_{g} \phi + \left( S_{g} + \beta \right) \phi & = \left( \eta_{1} + \beta \right) \phi \; {\rm in} \; M; \\
\left( \eta_{1} + \beta \right) \inf_{M} \phi & > 2^{p-2} \lambda_{\beta} \sup_{M} \phi^{p-1} \geqslant 2^{p-2} \lambda_{\beta} \phi^{p-1} > \lambda_{\beta} \phi^{p-1} \; {\rm in} \; M.
\end{split}
\end{equation}
Set
\begin{equation}\label{yamabe:eqns4}
\beta'  = \left( \eta_{1} + \beta \right) \sup_{M} \phi - 2^{p-2} \lambda_{\beta} \inf_{M} \phi^{p-1} > \left( \eta_{1} + \beta \right)\phi - 2^{p-2} \lambda_{\beta} \phi^{p-1} \; \text{pointwise}.
\end{equation}
Thus we have
\begin{equation}\label{yamabe:eqns5}
\begin{split}
-a\Delta_{g} \phi + \left(S_{g} + \beta \right) \phi & = \left( \eta_{1} + \beta \right) \phi > 2^{p-2} \lambda_{\beta} \phi^{p-1} > \lambda_{\beta} \phi^{p-1} \; {\rm in} \; M \; {\rm pointwise}, \\
\frac{\partial \phi}{\partial \nu} + \frac{2}{p - 2} h_{g} \phi & = 0 \; {\rm on} \; \partial M.
\end{split}
\end{equation}
\medskip

Now we construct the sub-solution by applying Proposition \ref{boundary:prop7}. Pick up a small enough interior Riemannian domain $ (\Omega, g) $ in which $ S_{g} < 0 $ such that the Dirichlet boundary value problem (\ref{boundary:eqn12}) with the given $ \lambda_{\beta} $ above has a positive solution $ u_{1} \in \calC_{0}(\Omega) \cap H_{0}^{1}(\Omega, g) $, i.e.
\begin{equation}\label{yamabe:eqn11a}
-a\Delta_{g} u_{1} + \left(S_{g} + \beta \right) u_{1} = \lambda_{\beta} u_{1}^{p-1} \; {\rm in} \; \Omega, u_{1} = 0 \; {\rm on} \; \partial M.
\end{equation}
Extend $ u_{1} $ by zero on the rest of $ \bar{M} $, we define
\begin{equation}\label{yamabe:eqn11}
u_{-} : = \begin{cases} u_{1}(x), & x \in \bar{\Omega} \\ 0, & \bar{M} \backslash \bar{\Omega} \end{cases}.
\end{equation}
Clearly $ u_{-} \in \calC_{0}(\bar{M}) $. Since $ u_{1} \in H_{0}^{1}(\Omega, g) $, $ u_{1} $ can be approximated by $ \lbrace v_{k} \rbrace \subset \calC_{c}^{\infty}(\Omega) $ in $ H^{1} $-sense. We extend $ v_{k} $ by zero to the rest of $ \bar{M} $ and the extensions are still $ \calC_{c}^{\infty}(\bar{M}) $. The extensions converge to the limit $ u_{-} $ in $ H^{1} $-sense, hence $ u_{-} \in H^{1}(M, g) $. Therefore $ u_{-} \in \calC_{0}(\bar{M}) \cap H^{1}(M, g) $. Define $ F(x, u) = -\left( S_{g} + \beta \right) u + \lambda u^{p-1} $ as in (\ref{yamabe:eqn4}), we observe that
\begin{equation}\label{yamabe:eqn12}
-a\Delta_{g} u_{-} \leqslant F(x, u_{-}) \; {\rm in} \; M, B_{g} u_{-} \leqslant 0 \; {\rm on} \; \partial M.
\end{equation}
\medskip

We construct the super-solution here. Pick up $ \gamma \ll 1 $ such that
\begin{equation}\label{yamabe:eqn10b}
0 < 20 \lambda \gamma + 2\gamma \cdot \sup_{M} \lvert S_{g} \rvert \gamma < \frac{\beta'}{2}, 31 \lambda (\phi + \gamma)^{p-2} \gamma < \frac{\beta'}{2}. 
\end{equation}
The choice of $ \gamma $ is dimensional specific. Set
\begin{align*}
V & = \lbrace x \in \Omega: u_{1}(x) > \phi(x) \rbrace, V' = \lbrace x \in \Omega: u_{1}(x) < \phi(x) \rbrace, D = \lbrace x \in \Omega : u_{1}(x) = \phi(x) \rbrace, \\
D' & = \lbrace x \in \Omega : \lvert u_{1}(x) - \phi(x) \rvert < \gamma \rbrace, D'' = \left\lbrace x \in \Omega : \lvert u_{1}(x) - \phi(x) \rvert > \frac{\gamma}{2} \right\rbrace.
\end{align*}
If $ \varphi \geqslant u_{1} $ pointwise, then $ \phi $ is a super-solution. If not, a good candidate of super-solution will be $ \max \lbrace u_{1}, \phi \rbrace $ in $ \Omega $ and $ \varphi $ outside $ \Omega $, this is an $ H^{1} \cap \calC_{0} $-function. Let $ \nu $ be the outward normal derivative of $ \partial V $ along $ D $. If $ \frac{\partial u_{1}}{\partial \nu} = - \frac{\partial \phi}{\partial \nu} $ on $ D $ then the super-solution has been constructed. However, this is in general not the case. If not, then $ \frac{\partial u_{1} - \partial \phi}{\partial \nu} \neq 0 $, which follows that $ 0 $ is a regular point of the function $ u_{1} - \varphi $ and hence $ D $ is a smooth submanifold of $ \Omega $. Define
\begin{equation}\label{yamabe:eqn13}
\Omega_{1} = V \cap D'', \Omega_{2} = V' \cap D'', \Omega_{3} = D'.
\end{equation}
Construct a specific smooth partition of unity $ \lbrace \chi_{i} \rbrace $ subordinate to $ \lbrace \Omega_{i} \rbrace $ as in Theorem 4.3 of \cite{XU3}, we define
\begin{equation}\label{yamabe:eqn13a}
\bar{u} = \chi_{1} u_{1} + \chi_{2} \phi + \chi_{3} \left( \phi + \gamma \right).
\end{equation}
Without loss of generality, we may assume that all $ \Omega_{i}, i = 1, 2, 3 $ are connected. Due to the same argument in Theorem 4.3 of \cite{XU3}, we conclude that $ \bar{u} \in \calC^{\infty}(\Omega) $ is a super-solution of the perturbed boundary Yamabe equation in $ \Omega $ pointwise, regardless of the boundary condition at the time being. By the definition of $ \bar{u} $, it is immediate that $ \bar{u} \geqslant u_{1} $. Define
\begin{equation}\label{yamabe:eqn13b}
u_{+} : = \begin{cases} \bar{u}, & \; {\rm in} \; \Omega; \\ \phi, & \; {\rm in} \; \bar{M} \backslash \Omega. \end{cases}
\end{equation}
It follows that $ u_{+} \in \calC^{\infty}(M) $ since $ \bar{u} = \phi $ near $ \partial \Omega $. By (\ref{yamabe:eqns5}) we conclude that
\begin{equation}\label{yamabe:eqn14}
-a\Delta_{g} u_{+} \geqslant F(x, u_{+}) \; {\rm in} \; M, B_{g} u_{+} \geqslant 0.
\end{equation}
Critically, $ 0 \leqslant u_{-} \leqslant u_{+} $ and $ u_{-} \not\equiv 0 $ on $ \bar{M} $. As discussed in Theorem \ref{yamabe:thm2}, we can then apply Theorem \ref{global:thm3}. With the aids of (\ref{yamabe:eqn12}) and (\ref{yamabe:eqn14}), we conclude that there exists a real, nonnegative solution $ u \in W^{2, p}(M, g) $ such that
\begin{equation*}
-a\Delta_{g} u + \left(S_{g} + \beta \right) u = \lambda_{\beta} u^{p-1} \; {\rm in} \; M, B_{g} u = 0 \; {\rm on} \; \partial M.
\end{equation*}
By bootstrapping method mentioned as above, we conclude that $ u \in \calC^{\infty}(\bar{M}) $. Now we show that $ u > 0 $. Let $ M = \max \lbrace S_{g} + \beta - \lambda_{\beta} u^{p-2}, 0 \rbrace $. It follows from above that
\begin{equation*}
-a\Delta_{g} u + M u \geqslant -a\Delta_{g} u + \left(S_{g} + \beta \right) u - \lambda_{\beta} u^{p-1} \geqslant 0.
\end{equation*}
Since $ u \in \calC^{\infty}(M) $ it is smooth locally, then local strong maximum principle says that if $ u = 0 $ in some interior domain $ \Omega $ then $ u \equiv 0 $ on $ \Omega $, a continuation argument then shows that $ u \equiv 0 $ in $ M $. But $ u \geqslant u_{-} $ and $ u_{-} > 0 $ within some region. Thus $ u > 0 $ in the interior $ M $. By the same argument in \cite[\S1]{ESC}, we conclude that $ u > 0 $ on $ \bar{M} $.
\end{proof}
\medskip

\begin{theorem}\label{yamabe:thms}
Let $ (\bar{M}, g) $ be a compact manifold with boundary. Assume the scalar curvature $ S_{g} < 0 $ somewhere on $ M $ and the first eigenvalue $ \eta_{1} > 0 $. Then there exists some $ \lambda > 0 $ such that the Yamabe equation (\ref{yamabe:eqn1}) has a real, positive, smooth solution.
\end{theorem}
\begin{proof} By Theorem \ref{yamabe:thm5}, we have a sequence of real, positive, smooth solutions $ \lbrace u_{\beta} \rbrace $ when $ \beta < 0 $ and $ \lvert \beta \rvert $ is small enough, i.e.
\begin{equation}\label{yamabe:eqns6}
-a\Delta_{g} u_{\beta} + \left(S_{g} + \beta\right) u_{\beta} = \lambda_{\beta} u_{\beta}^{p-1} \; {\rm in} \; M, \frac{\partial u_{\beta}}{\partial \nu} + \frac{2}{p - 2} h_{g} u_{\beta} = 0 \; {\rm on} \; \partial M.
\end{equation}
We show first that $ \lbrace \lambda_{\beta} \rbrace $ is bounded above, and is increasing and continuous when $ \beta \rightarrow 0^{-} $. We may assume $ \int_{M} d\omega = 1 $ for this continuity verification, since otherwise only an extra term with respect to $ \text{Vol}_{g} $ will appear. Recall that
\begin{equation*}
\lambda_{\beta} = \inf_{u \neq 0, u \in H^{1}(M)} \left\lbrace \frac{\int_{M} a\lvert \nabla_{g} u \rvert^{2} d\omega + \int_{M} \left( S_{g} + \beta \right) u^{2} d\omega + \int_{\partial M} \frac{2a}{p-2}h_{g} u^{2} dS}{\left( \int_{M} u^{p} d\omega \right)^{\frac{2}{p}}} \right\rbrace.
\end{equation*}
It is immediate that if $ \beta_{1} < \beta_{2} < 0 $ then $ \lambda_{\beta_{1}} \leqslant \lambda_{\beta_{2}} $. For continuity we assume $ 0 < \beta_{2} - \beta_{1} < \gamma $. For each $ \epsilon > 0 $, there exists a function $ u_{0} $ such that
\begin{equation*}
\frac{\int_{M} a\lvert \nabla_{g} u_{0} \rvert^{2} d\omega + \int_{M} \left( S_{g} + \beta_{1} \right) u_{0}^{2} d\omega + \int_{\partial M} \frac{2a}{p-2}h_{g} u_{0}^{2} dS}{\left( \int_{M} u_{0}^{p} d\omega \right)^{\frac{2}{p}}} < \lambda_{\beta_{1}} + \epsilon.
\end{equation*}
It follows that
\begin{align*}
\lambda_{\beta_{2}} & \leqslant \frac{\int_{M} a\lvert \nabla_{g} u_{0} \rvert^{2} d\omega + \int_{M} \left( S_{g} + \beta_{2} \right) u_{0}^{2} d\omega + \int_{\partial M} \frac{2a}{p-2}h_{g} u_{0}^{2} dS}{\left( \int_{M} u_{0}^{p} d\omega \right)^{\frac{2}{p}}} \\
& \leqslant \frac{\int_{M} a\lvert \nabla_{g} u_{0} \rvert^{2} d\omega + \int_{M} \left( S_{g} + \beta_{1} \right) u_{0}^{2} d\omega + \int_{\partial M} \frac{2a}{p-2}h_{g} u_{0}^{2} dS}{\left( \int_{M} u_{0}^{p} d\omega \right)^{\frac{2}{p}}} + \frac{\left( \beta_{2} - \beta_{1} \right) \int_{M} u_{0}^{2} d\omega}{\left( \int_{M} u_{0}^{p} d\omega \right)^{\frac{2}{p}}} \\
& \leqslant \lambda_{\beta_{1}} + \epsilon + \beta_{2} - \beta_{1} < \lambda_{\beta_{1}} + \epsilon + \beta_{2} - \beta_{1}.
\end{align*}
Since $ \epsilon $ is arbitrarily small, we conclude that
\begin{equation*}
0 < \beta_{2} - \beta_{1} < \gamma \Rightarrow \lvert \lambda_{\beta_{2}} - \lambda_{\beta_{1}} \rvert \leqslant 2\gamma.
\end{equation*}
By equation (4) in \cite[\S1]{ESC} , we conclude that
\begin{equation*}
\lambda_{\beta} \leqslant Q(\mathbb{S}_{+}^{n})
\end{equation*}
 Fix some $ \beta_{0} < 0 $ with $ \lambda_{ \beta_{0}} > 0 $, we have
\begin{equation}\label{yamabe:eqns7}
\lambda_{\beta_{0}} \leqslant \lambda_{\beta} \leqslant Q(\mathbb{S}_{+}^{n}), \forall \beta \in [\beta_{0}, 0], \lim_{\beta \rightarrow 0^{-}} \lambda_{\beta} : = \lambda.
\end{equation}
Next we show that for some $ r > p $, 
\begin{equation}\label{yamabe:eqns8}
\lVert u_{\beta} \rVert_{\calL^{p}(M, g)} \geqslant \mathcal{K}_{3} > 0, \lVert u_{\beta} \rVert_{\calL^{r}(M, g)} \leqslant C, \forall \beta \in [\beta_{0}, 0).
\end{equation}
For the lower bound of $ \calL^{p} $-norm, we pair $ u_{\beta} $ on both sides of (\ref{yamabe:eqns6}),
\begin{equation*}
a\lVert \nabla_{g} u_{\beta} \rVert_{\calL^{2}(M, g)}^{2} + \int_{M} \left(S_{g} + \beta \right) u_{\beta}^{2} d\omega + \int_{\partial M} \frac{2a}{p-2} h_{g} u_{\beta}^{2} dS = \lambda_{\beta} \lVert u_{\beta} \rVert_{\calL^{p}(M, g)}^{p}, \forall \beta \in [\beta_{0}, 0).
\end{equation*}
By characterization of $ \lambda_{\beta} $,
\begin{align*}
& \lambda_{\beta} \leqslant \frac{\int_{M} a\lvert \nabla_{g} u_{\beta} \rvert^{2} d\omega + \int_{M} \left( S_{g} + \beta \right) u_{\beta}^{2} d\omega + \int_{\partial M} \frac{2a}{p-2}h_{g} u_{\beta}^{2} dS}{\left( \int_{M} u_{\beta}^{p} d\omega \right)^{\frac{2}{p}}} = \lambda_{\beta} \cdot \frac{\lVert u_{\beta} \rVert_{\calL^{p}(M, g)}^{p}}{\left( \int_{M} u_{\beta}^{p} d\omega \right)^{\frac{2}{p}}} \\
\Rightarrow & \lambda_{\beta} \leqslant \lambda_{\beta} \lVert u_{\beta} \rVert_{\calL^{p}(M, g)}^{p - 2}.
\end{align*}
Thus the lower bound in (\ref{yamabe:eqns8}) holds. For the upper bound of $ \calL^{r} $-norm we need local analysis. Denote the local solutions of (\ref{boundary:eqn12}) by $ \lbrace \tilde{u}_{\beta} \rbrace $, i.e.
\begin{equation}\label{yamabe:eqns9a}
-a\Delta_{g} \tilde{u}_{\beta} + \left(S_{g} + \beta\right) \tilde{u}_{\beta} = \lambda_{\beta} \tilde{u}_{\beta}^{p-1} \; {\rm in} \; \Omega, \tilde{u}_{\beta} = 0 \; {\rm on} \; \partial \Omega
\end{equation}
with fixed domain $ \Omega $. Recall the construction of super-solution of each $ u_{\beta} $ in Theorem \ref{yamabe:thm5}, we have
\begin{equation*}
0 \leqslant u_{-, \beta} \leqslant u_{\beta} \leqslant u_{+, \beta} = \begin{cases} \bar{u}_{\beta}, & {\rm in} \; \Omega \\ \phi, & {\rm in} \; M \backslash \Omega \end{cases}.
\end{equation*}
where $ \bar{u}_{\beta} $ is of the form
\begin{equation*}
\bar{u}_{\beta} = \chi_{1} \tilde{u}_{\beta} + \chi_{2} \phi + \chi_{3} (\phi + \gamma).
\end{equation*} 
Thus it suffices to show that
\begin{equation*}
\lVert \tilde{u}_{\beta} \rVert_{\calL^{r}(\Omega, g)} \leqslant C_{1}, r> p, \forall \beta \in [\beta_{0}, 0).
\end{equation*}
Pairing $ \tilde{u}_{\beta} $ on both sides of (\ref{yamabe:eqns9a}),
\begin{equation}\label{yamabe:eqns9}
\begin{split}
a\lVert \nabla_{g} \tilde{u}_{\beta} \rVert_{\calL^{2}(\Omega, g)}^{2} & =  \lambda_{\beta} \lVert \tilde{u}_{\beta} \rVert_{\calL^{p}(\Omega, g)}^{p} - \int_{M} \left( S_{g} + \beta \right) u_{\beta}^{2} \dvol \\
\Rightarrow \lambda_{\beta} \lVert \tilde{u}_{\beta} \rVert_{\calL^{p}(\Omega, g)}^{p} & \leqslant a\lVert \nabla_{g} \tilde{u}_{\beta} \rVert_{\calL^{2}(\Omega, g)}^{2} + \left( \sup_{M} \lvert S_{g} \rvert + \lvert \beta \rvert \right) \lVert u_{\beta} \rVert_{\calL^{2}(\Omega, g)}^{2}.
\end{split}
\end{equation}
Recall the functional
\begin{equation*}
J(u) = \int_{\Omega} \left( \frac{1}{2} \sum_{i, j} a_{ij}(x) \partial_{i}u \partial_{j} u - \frac{\lambda_{\beta} \sqrt{\det(g)} }{p} u^{p} - \frac{1}{2} \left(S_{g} + \beta \right) u^{2} \sqrt{\det(g)} \right) dx
\end{equation*}
and the constant $ K_{0} $ in \cite[\S3]{XU3}, which depends on $ \lambda_{\beta} \in [\lambda_{\beta_{0}}, Q(\mathbb{S}_{+}^{n})] $ only. Due to Theorem 1.1 of \cite{WANG}, each solution $ \tilde{u}_{\beta} $ satisfies
\begin{equation}\label{yamabe:eqns10}
J(\tilde{u}_{\beta}) \leqslant K_{0} \Rightarrow \frac{a}{2} \lVert \nabla_{g} \tilde{u}_{\beta} \rVert_{\calL^{2}(\Omega, g)}^{2} - \frac{\lambda_{\beta}}{p} \lVert \tilde{u}_{\beta} \rVert_{\calL^{p}(\Omega, g)}^{p} - \frac{1}{2} \int_{M} \left(S_{g} + \beta \right) \tilde{u}_{\beta}^{2} \dvol \leqslant K_{0}.
\end{equation}
Let $ \lambda_{1} $ be the first eigenvalue of $ -\Delta_{g} $ on $ \Omega $ with Dirichlet boundary condition. Apply the estimate (\ref{yamabe:eqns9}) into (\ref{yamabe:eqns10}), we have
\begin{align*}
\frac{a}{2} \lVert \nabla_{g} \tilde{u}_{\beta} \rVert_{\calL^{2}(\Omega, g)}^{2} & \leqslant K_{0} + \frac{1}{p} \left(a\lVert \nabla_{g} \tilde{u}_{\beta} \rVert_{\calL^{2}(\Omega, g)}^{2} + \left( \sup_{M} \lvert S_{g} \rvert + \lvert \beta \rvert \right) \lVert u_{\beta} \rVert_{\calL^{2}(\Omega, g)}^{2} \right) \\
& \qquad + \frac{1}{2} \left( \sup_{M} \lvert S_{g} \rvert + \lvert \beta \rvert \right) \lVert u_{\beta} \rVert_{\calL^{2}(\Omega, g)}^{2} \\
& \leqslant K_{0} + \frac{a(n - 2)}{2n} \lVert \nabla_{g} \tilde{u}_{\beta} \rVert_{\calL^{2}(\Omega, g)}^{2} \\
& \qquad + \left( \frac{n - 2}{2n} + \frac{1}{2} \right) \left( \sup_{M} \lvert S_{g} \rvert + \lvert \beta \rvert \right) \cdot \lambda_{1}^{-1} \lVert \nabla_{g} \tilde{u}_{\beta} \rVert_{\calL^{2}(\Omega, g)}^{2}; \\
\Rightarrow & \left(\frac{a}{n} - \left( \frac{n - 2}{2n} + \frac{1}{2} \right) \left( \sup_{M} \lvert S_{g} \rvert + \lvert \beta \rvert \right) \cdot \lambda_{1}^{-1} \right) \lVert \nabla_{g} \tilde{u}_{\beta} \rVert_{\calL^{2}(\Omega, g)}^{2} \leqslant K_{0}.
\end{align*}
Recall in Remark \ref{boundary:re2} in which we have chosen $ \Omega $ small enough so that 
\begin{equation*}
\frac{a}{n} - \left( \frac{n - 2}{2n} + \frac{1}{2} \right) \left( \sup_{M} \lvert S_{g} \rvert + \lvert \beta \rvert \right) \cdot \lambda_{1}^{-1} > 0,
\end{equation*}
which holds for all $ \beta \in [\beta_{0}, 0) $. It follows from above that there exists a constant $ C_{0}' $ such that
\begin{equation*}
\lVert \nabla_{g} \tilde{u}_{\beta} \rVert_{\calL^{2}(\Omega, g)}^{2} \leqslant C_{0}', \forall \beta \in [\beta_{0}, 0].
\end{equation*}
Apply (\ref{yamabe:eqns9}) with the other way around, we conclude that
\begin{align*}
\lambda_{\beta} \lVert \tilde{u}_{\beta} \rVert_{\calL^{p}(\Omega, g)}^{p} & \leqslant a\lVert \nabla_{g} \tilde{u}_{\beta} \rVert_{\calL^{2}(\Omega, g)}^{2} + \left( \sup_{M} \lvert S_{g} \rvert + \lvert \beta \rvert \right)  \lVert \tilde{u}_{\beta} \rVert_{\calL^{2}(\Omega, g)}^{2} \\
& \leqslant \left( a + \left( \sup_{M} \lvert S_{g} \rvert + \lvert \beta \rvert \right)  \lambda_{1}^{-1} \right) \lVert \nabla_{g} u_{\beta} \rVert_{\calL^{2}(\Omega, g)}^{2}.
\end{align*}
We conclude that
\begin{equation}\label{yamabe:eqns11}
\lVert \tilde{u}_{\beta} \rVert_{\calL^{p}(\Omega, g)} \leqslant C_{1}, \forall \beta \in [\beta_{0}, 0].
\end{equation}
Note that this uniform upper bound $ C_{1} $ is unchanged if we further shrink the domain $ \Omega $. Note that this shrinkage of domain is a restriction, not a scaling of domain or metric. We can then, without loss of generality, assume that $ C_{1} = 1 $. This can be done by scaling the metric one time, uniformly for all $ \beta \in [\beta_{0}, 0) $. Note that this scaling does not affect the local solvability in Proposition \ref{boundary:prop7} since the estimates in Appendix A of \cite{XU3} still hold under scaling. After a one-time scaling $ g \mapsto \delta g $ we still have $ \lambda_{\beta} \in [\lambda_{\beta_{0}}, Q(\mathbb{S}_{+}^{n})] $ due to the characterization of $ \lambda_{\beta} $, if $ \delta < 1 $. Since $ \beta < 0 $, the lower bound of $ \lambda_{\beta_{0}} $ is unchanged. We still denote the new metric by $ g $, which follows that
\begin{equation}\label{yamabe:eqns12}
\lVert \tilde{u}_{\beta} \rVert_{\calL^{p}(\Omega, g)} \leqslant 1, \forall \beta \in [\beta_{0}, 0).
\end{equation}
According to equation (4) of \cite[\S1]{ESC}, we have
\begin{equation}\label{yamabe:eqns13}
\lambda_{\beta} \leqslant Q(\mathbb{S}_{+}^{n}) = \frac{n(n - 2)}{4} \text{Vol}\left(\mathbb{S}_{+}^{n}\right)^{\frac{2}{n}} = 2^{-\frac{2}{n}} \frac{n(n - 2)}{4} \text{Vol}\left(\mathbb{S}^{n}\right)^{\frac{2}{n}} = 2^{-\frac{2}{n}} aT.
\end{equation}
We point out that the ratio $ \frac{\lambda_{\beta}}{aT} < 1 $ still holds after one-time scaling. Due to the idea of Trudinger, Aubin and the argument in Theorem 4.4 of \cite{XU3}, we pair $ \tilde{u}_{\beta}^{1 + 2 \delta} $ for some $ \delta > 0 $ on both sides of (\ref{yamabe:eqns9a}) and denote $ w_{\beta} = \tilde{u}_{\beta}^{1 + \delta} $, we have
\begin{align*}
& \int_{\Omega} a \nabla_{g} \tilde{u}_{\beta} \cdot \nabla_{g} \left(\tilde{u}_{\beta}^{1 + 2\delta} \right) \dvol + \int_{\Omega} \left(S_{g} + \beta \right) u_{\beta}^{2 + 2\delta} \dvol = \lambda_{\beta} \int_{\Omega} \tilde{u}_{\beta}^{p + 2\delta} \dvol; \\
\Rightarrow & \frac{1 + 2\delta}{1 + \delta^{2}} \int_{\Omega} a \lvert \nabla_{g} w_{\beta} \rvert^{2} \dvol = \lambda_{\beta} \int_{\Omega} w_{\beta}^{2} \tilde{u}_{\beta}^{p-2} \dvol - \int_{\Omega} \left(S_{g} + \beta \right) w_{\beta}^{2} \dvol.
\end{align*}
When the radius $ r $ of $ \Omega $ is small enough, there exists a constant $ A $ such that
\begin{equation*}
\lVert u \rVert_{\calL^{p}(\Omega, g)}^{2} \leqslant (1 + Ar^{2}) \lVert u \rVert_{\calL^{p}(\Omega)}^{2}, \lVert D u \rVert_{\calL^{2}(\Omega)}^{2} \leqslant (1 + Ar^{2}) \lVert \nabla_{g} u \rVert_{\calL^{2}(\Omega, g)}^{2}.
\end{equation*}
Due to standard sharp Sobolev embedding on Euclidean space, we have
\begin{align*}
\lVert w_{\beta} \rVert_{\calL^{p}(\Omega, g)}^{2} & \leqslant (1 + Ar^{2}) \lVert w_{\beta} \rVert_{\calL^{p}(\Omega)}^{2} \leqslant \frac{1 + Ar^{2}}{T} \lVert D w_{\beta} \rVert_{\calL^{2}(\Omega)}^{2} \leqslant \frac{\left( 1 + Ar^{2} \right)^{2}}{T} \lVert \nabla_{g} w_{\beta} \rVert_{\calL^{2}(\Omega, g)}^{2} \\
& = \frac{\left( 1 + Ar^{2} \right)^{2}}{aT} \cdot \frac{1 + \delta^{2}}{1 + 2\delta} \left(\lambda_{\beta} \int_{\Omega} w_{\beta}^{2} \tilde{u}_{\beta}^{p-2} \dvol - \int_{\Omega} \left(S_{g} + \beta \right) w_{\beta}^{2} \dvol \right) \\
& \leqslant \frac{\left( 1 + Ar^{2} \right)^{2}}{aT} \cdot \frac{1 + \delta^{2}}{1 + 2\delta} \lambda_{\beta} \lVert w_{\beta} \rVert_{\calL^{p}(\Omega, g)}^{2} \lVert \tilde{u}_{\beta} \rVert_{\calL^{p}(\Omega, g)}^{p - 2} + C_{\beta} \lVert w_{\beta} \rVert_{\calL^{2}(\Omega, g)}^{2} \\
& \leqslant \left( 1 + Ar^{2} \right)^{2} \cdot \frac{1 + \delta^{2}}{1 + 2\delta} \cdot \frac{2^{-\frac{2}{n}} aT}{aT} \lVert w_{\beta} \rVert_{\calL^{p}(\Omega, g)}^{2} + C_{\beta} \lVert w_{\beta} \rVert_{\calL^{2}(\Omega, g)}^{2}
\end{align*}
by H\"older's inequality and (\ref{yamabe:eqns13}). Note that $ C_{\beta} $ is uniformly bounded above for all $ \beta \in [\beta_{0}, 0) $. Due to the last line above, we can choose $ r, \delta $ small enough so that
\begin{equation*}
\left( 1 + Ar^{2} \right)^{2} \cdot \frac{1 + \delta^{2}}{1 + 2\delta} \cdot \frac{2^{-\frac{2}{n}} aT}{aT} < 1.
\end{equation*}
It follows that
\begin{equation*}
\lVert w_{\beta} \rVert_{\calL^{p}(\Omega, g)}^{2} \leqslant \mathcal{K}_{1} \lVert w_{\beta} \rVert_{\calL^{2}(\Omega, g)}^{2}.
\end{equation*}
Recall that $ w_{\beta} = \tilde{u}_{\beta}^{1 + \delta} $. Applying H\"older's inequality on right side above, and note that $ \text{Vol}_{g}(\Omega) \leqslant \text{Vol}_{g}(M) $, we conclude by exactly the same argument as in \cite[Prop.~4.4]{PL}, \cite[Thm.~4.4]{XU3} that
\begin{equation}\label{yamabe:eqns14}
\lVert \tilde{u}_{\beta} \rVert_{\calL^{r}(\Omega, g)} \leqslant \mathcal{K}_{2}, r = p(1 + \delta), \forall \beta \in [\beta_{0}, 0).
\end{equation}
Recall in determining $ \phi $ we require
\begin{equation*}
\frac{\eta_{1} + \beta}{2^{p-2}\lambda_{\beta}} > \delta^{p-2} \frac{\sup_{M} \varphi^{p-1}}{\inf_{M} \varphi}.
\end{equation*}
Since when $ \beta \in [\beta_{0}, 0] $, we have $ \lambda_{\beta} \in [\lambda_{\beta_{0}}, Q(\mathbb{S}_{+}^{n})] $, we can choose a fixed scaling $ \delta $ for all $ \beta \in [\beta_{0}, 0) $, thus by Minkowski inequality and the construction of super-solutions, 
\begin{equation*}
\lVert u_{\beta} \rVert_{\calL^{r}(M, g)} \leqslant \lVert u_{+, \beta} \rVert_{\calL^{r}(M, g)} \leqslant A_{1} \left( \lVert \tilde{u}_{\beta} \rVert_{\calL^{r}(\Omega, g)} + \lVert \phi \rVert_{\calL^{r}(M, g)} \right) : = C, \forall \beta \in [\beta_{0}, 0).
\end{equation*}
By repeated elliptic regularities and Sobolev embedding, uniform boundedness in $ \calL^{r} $-norm implies that
\begin{equation}\label{yamabe:eqns15}
\lVert u_{\beta} \rVert_{\calC^{2, \alpha}(M)} \leqslant \mathcal{K}_{0}, \forall \beta \in [\beta_{0}, 0).
\end{equation}
By Arzela-Ascoli, it follows that up to a subsequence, $ \lim_{\beta \rightarrow 0^{-}} u_{\beta} = u $. Due to (\ref{yamabe:eqns7}), we have $ \lim_{\beta \rightarrow 0^{-}} \lambda_{\beta} = \lambda $. It follows that the limiting function $ u $ satisfies
\begin{align*}
-a\Delta_{g} u + S_{g} u = \lambda u^{p-1} \; {\rm in} \; M; \\
\frac{\partial u}{\partial \nu} + \frac{2}{p-2} h_{g} u = 0 \; {\rm on} \; \partial M.
\end{align*}
By \cite{Che} we conclude that $ u \in \calC^{\infty}(M) $. Lastly we show $ u > 0 $. Clearly $ u \geqslant 0 $ since $ u_{\beta} > 0 $. By (\ref{yamabe:eqns8}) we conclude that $ \lVert u_{\beta} \rVert_{\calL^{p}(M, g)} \geqslant \mathcal{K}_{3} > 0, \forall \beta \in [\beta_{0}, 0) $. By Arzela-Ascoli again, up to a subsequence,
\begin{equation*}
0 < \mathcal{K}_{3} \leqslant \lim_{\beta \rightarrow 0^{-}} \lVert u_{\beta} \rVert_{\calL^{p}(M, g)} = \lVert u \rVert_{\calL^{p}(M, g)}.
\end{equation*}
Thus by maximum principle, $ u > 0 $ in the interior $ M $. By the same argument in \cite[\S1]{ESC}, we conclude that $ u > 0 $ on $ \bar{M} $.
\end{proof}
\medskip

As discussed before, in general $ h_{g} > 0 $ everywhere on $ \partial M $ is not the case. Similarly when $ \eta_{1} > 0 $, it is not always possible that $ S_{g} < 0 $ somewhere. The next two result, analogous to Theorem 4.5 in \cite{XU3}, shows the existence of metric $ \tilde{g} $ under conformal change such that $ \tilde{S} $ is negative somewhere and the sign of $ \tilde{h} $ will be the same as the sign of $ h_{g} $ pointwise, provided that $ S_{g} \geqslant 0 $ everywhere. Note again that by Proposition \ref{boundary:prop1}, the signs of first eigenvalues keep same under conformal change.
\begin{theorem}\label{yamabe:thm6}
Let $ (\bar{M}, g) $ be a compact manifold with smooth boundary. Let $ S_{g} \geqslant 0 $ everywhere. There exists a conformal metric $ \tilde{g} $ associated with scalar curvature $ \tilde{S} $ and mean curvature $ \tilde{h} $ such that $ \tilde{S} < 0 $ somewhere, and $ \text{sgn}(h_{g}) = \text{sgn}(\tilde{h}) $ pointwise on $ \partial M $.
\end{theorem}
\begin{proof} By scaling we can, without loss of generality, assume that $ \lvert S_{g} \rvert \leqslant 1 $ on $ \bar{M} $.
Based on exactly the same construction in Theorem 4.5 of \cite{XU3}, there exists a smooth function $ F \in \calC^{\infty}(\bar{M}) $ such that (i) $ \int_{M} F d\omega = 0 $; (ii) $ F $ is very negative at some interior point $ p \in M $; (iii) $ \lVert F \rVert_{H^{s - 2}(M, g)} $ is small enough, here $ s = \frac{n}{2} + 1 $ if $ n $ is even and $ s = \frac{n + 1}{2} $ if $ n $ is odd. The largeness and smallness will be determined later. Consider the following linear PDE with Neumann boundary condition
\begin{equation}\label{yamabe:eqn15}
-a\Delta_{g} u' = F \; {\rm in} \; M, \frac{\partial u'}{\partial \nu} = 0 \; {\rm on} \; \partial M.
\end{equation}
By standard elliptic theory, we conclude that there exists $ u' \in \calC^{\infty}(\bar{M}) $ solves (\ref{yamabe:eqn15}) uniquely up to constants. By a standard elliptic regularity \cite[Prop.~7.4]{T}, we conclude that
\begin{equation*}
\lVert u' \rVert_{H^{s}(M, g)} \leqslant C^{**}\left( \lVert F \rVert_{H^{s - 2}(M, g)} + \lVert u' \rVert_{\calL^{2}(M, g)} \right).
\end{equation*}
Pairing both side of (\ref{yamabe:eqn15}) by $ u' $, we conclude that
\begin{equation*}
\lVert u' \rVert_{\calL^{2}(M, g)} \leqslant C_{0} \lVert \nabla_{g} u' \rVert_{\calL^{2}(M, g)} \leqslant C_{0} \lVert F \rVert_{\calL^{2}(M, g)}.
\end{equation*}
This can be done by taking $ u' \mapsto u' + \epsilon_{1} $ so that $ \int_{M} (u' + \epsilon_{1} ) d\omega = 0 $. Since $ u' + \epsilon_{1} $ also solves (\ref{yamabe:eqn15}) we assume without loss of generality that $ \int_{M} u' d\omega = 0 $ thus the Poincar\'e inequality holds. Therefore there exists some $ C_{1} $ such that
\begin{equation*}
\sup_{\bar{M}} \lvert u \rvert \leqslant C_{1} \lVert F \rVert_{H^{s - 2}(M, g)}
\end{equation*}
with $ s = \frac{n}{2} + 1 $ when $ n $ is even or $ s = \frac{n + 1}{2} $ when $ n $ is odd. Then, same as in Theorem 4.4 of \cite{XU3}, pick up some $ C > 1 $ we choose can choose $ F $ such that
\begin{equation*}
F(q) \leqslant - \frac{C}{2}, \lvert u' \rvert \leqslant \frac{C}{8} \; {\rm in} \; \bar{M}.
\end{equation*}
Finally we choose
\begin{equation}\label{yamabe:eqn16}
u : = u' + \frac{C}{4}.
\end{equation}
It follows that this positive function $ u \in \left[ \frac{C}{8}, \frac{3C}{8} \right], u \in \calC^{\infty}(\bar{M}) $ and $ u $ solves (\ref{yamabe:eqn15}) since constant functions are in the kernel of $ -a\Delta_{g} $ with Neumann boundary condition. Using the choice of $ u $ in (\ref{yamabe:eqn16}), we define
\begin{equation}\label{yamabe:eqn17}
\begin{split}
\tilde{h} & = \frac{p - 2}{2} u^{\frac{2}{p}} \left( \frac{\partial u}{\partial \nu} + \frac{2}{p - 2} h_{g} u \right) \; {\rm on} \; \partial M; \\
\tilde{S} & = u^{1 - p} \left( -a\Delta_{g} u + S_{g} u \right) \; {\rm in} \; M.
\end{split}
\end{equation}
The first line in (\ref{yamabe:eqn17}) says
\begin{equation*}
\tilde{h} = \frac{p - 2}{2} u^{\frac{2}{p}} \left( 0 + \frac{2}{p - 2} h_{g} u \right) =  u^{\frac{2}{p} + 1} h_{g}.
\end{equation*}
Since $ u > 0 $, hence at each point of $ \partial M $, the sign of $ \tilde{h} $ is the same as the sign of $ h_{g} $. From the second line of (\ref{yamabe:eqn17}), we see that at the point  $ q $, 
\begin{equation*}
\tilde{S}(q) = u(q)^{1-p} \left( F(q) + S_{g}(q) u_{q} \right) \leqslant u(q)^{1 - p} \left( F(q) + \sup_{\bar{M}} \lvert S_{g} \rvert \sup_{\bar{M}} \lvert u \rvert \right) \leqslant u(q)^{1 - p} \left(-\frac{C}{2} + \frac{3C}{8} \right) < 0.
\end{equation*}
Lastly, we notice that since a real, positive $ u \in \calC^{\infty}(\bar{M}) $ solves the boundary value problem (\ref{yamabe:eqn17}), there exists a conformal metric $ \tilde{g} = u^{p-2} g $ associated with $ \tilde{S} $ and $ \tilde{h} $, where $ \tilde{S} $ and $ \tilde{h} $ has desired properties.
\end{proof}
\medskip

There is an immediate consequence with respect to some metric $ g $ associate with $ S_{g} \leqslant 0 $ everywhere.
\begin{corollary}\label{yamabe:cor3}
Let $ (\bar{M}, g) $ be a compact manifold with smooth boundary. Let $ S_{g} \leqslant 0 $ everywhere. There exists a conformal metric $ \tilde{g} $ associated with scalar curvature $ \tilde{S} $ and mean curvature $ \tilde{h} $ such that $ \tilde{S} > 0 $ somewhere, and $ \text{sgn}(h_{g}) = \text{sgn}(\tilde{h}) $ pointwise on $ \partial M $.
\end{corollary}
\begin{proof} Choosing $ F' = -F $ as above and everything follows exactly the same as in Theorem \ref{yamabe:thm5}.
\end{proof}
\medskip

With the help of Theorem \ref{yamabe:thm3} and \ref{yamabe:thm5}, we can handle the general case when $ \eta_{1} > 0 $ with arbitrary $ h_{g} $ and $ S_{g} $.
\begin{theorem}\label{yamabe:thm7}
Let $ (\bar{M}, g) $ be a compact manifold with smooth boundary. When $ \eta_{1} > 0 $, there exists some $ \lambda > 0 $ such that (\ref{yamabe:eqn1}) has a real, positive solution $ u \in \calC^{\infty}(\bar{M}) $.
\end{theorem}
\begin{proof} Start with the metric $ g $, we discuss the solvability in three cases:

(i) When $ S_{g} < 0 $ somewhere and $ h_{g} > 0 $ everywhere on $ \partial M $, this is exactly Theorem \ref{yamabe:thm5}.

(ii) When $ S_{g} \geqslant 0 $ everywhere and $ h_{g} > 0 $ everywhere on $ \partial M $. In this case, we apply Theorem \ref{yamabe:thm6} and find out some $ \tilde{g}_{1} = u^{p-2} g $ such that $ \tilde{S}_{1} < 0 $ somewhere and $ \tilde{h} > 0 $ everywhere; then by Theorem \ref{yamabe:thm5}, there exists some real, positive $ v \in \calC^{\infty}(\bar{M}) $ such that $ \tilde{g} = v^{p-2} \tilde{g}_{1} $ associates with constant scalar curvature and the boundary is minimal. Therefore,
\begin{equation*}
\tilde{g} = v^{p-2} \tilde{g}_{1} = v^{p-2} u^{p-2} g = (uv)^{p-2} g
\end{equation*}
is the desired conformal change.

(iii) When either $ S_{g} < 0 $ somewhere or $ S_{g} \geqslant 0 $ everywhere, and $ h_{g} $ changes sign on $ \partial M $. In this case, we first apply Theorem \ref{yamabe:thm3} and obtain some $ \tilde{g}_{1} = u^{p-2} g $ with $ \tilde{h}_{1} > 0 $ everywhere. If $ \tilde{S}_{1} < 0 $ somewhere then we apply Theorem \ref{yamabe:thm5} directly to get some metric with constant scalar curvature, two steps. If not, i.e. $ \tilde{S}_{1} \geqslant 0 $ everywhere, then we apply Theorem \ref{yamabe:thm6} to get some $ \tilde{g}_{2} = v^{p-2} \tilde{g}_{1} $ such that $ \tilde{S}_{2} < 0 $ somewhere and we still have $ \tilde{h}_{2} > 0 $ everywhere on $ \partial M $. Lastly, applying Theorem \ref{yamabe:thm5} and thus we get $ \tilde{g}_{3} = w^{p-2} \tilde{g}_{2} $ which admits a constant scalar curvature with minimal boundary. Therefore
\begin{equation*}
\tilde{g}_{3} = w^{p-2} \tilde{g}_{2}= w^{p-2}v^{p-2} \tilde{g}_{1} = w^{p-2}v^{p-2}u^{p-2} g = (uvw)^{p-2} g
\end{equation*}
has the desired property.
\end{proof}

\section*{Acknowledgement}
The author would like to thank his advisor Prof. Steven Rosenberg for his great support and mentorship. The author also owes thanks to Prof. Richard Melrose for learning many PDE and geometric skills from his papers, books, notes, and especially many courses taken in MIT instructed by him.

\bibliographystyle{plain}
\bibliography{Yamabess}

\end{document}